\documentclass[a4paper,12pt]{amsart}

\usepackage{amssymb,amsmath,amsthm,mathrsfs,enumerate,graphicx, color}
\usepackage[pdfpagelabels,colorlinks,linkcolor=MidnightBlue,citecolor=OliveGreen,urlcolor=green]{hyperref}
\usepackage[dvipsnames]{xcolor}
\usepackage{stix}
\usepackage{esint}

\usepackage{slashbox}
\usepackage[abbrev, nobysame]{amsrefs}

\newtheorem{thm}{Theorem}
\newtheorem*{thm*}{Theorem}

\theoremstyle{plain}
\newtheorem{Theorem}{\bf Theorem}[section]
\newtheorem{Lemma}{\bf Lemma}[section]
\newtheorem{Proposition}{\bf Proposition}[section]
\newtheorem{Corollary}{\bf Corollary}[section]
\newtheorem{Remark}{\bf Remark}[section]
\newtheorem{Example}{\bf Example}[section]
\newtheorem{Definition}{\bf Definition}[section]

\newenvironment{lemma}{\begin{Lemma}$\!\!\!$}{\end{Lemma}}
\newenvironment{proposition}{\begin{Proposition}$\!\!\!$}{\end{Proposition}}
\newenvironment{corollary}{\begin{Corollary}$\!\!\!$}{\end{Corollary}}
\newenvironment{remark}{\begin{Remark}$\!\!\!$}{\end{Remark}}

\newenvironment{definition}{\begin{Definition}$\!\!\!$}{\end{Definition}}

\numberwithin{equation}{section}

\newcommand{\lesi}{\lesssim}

\newcommand{\supp}{\operatorname{supp}}

\newcommand{\f}{\frac}

\newcommand{\vc}{\infty}

\newcommand{\di}{\mathop{}\!\mathrm{d}}

\newcommand{\HH}{\mathbb H}
\newcommand{\dist}{\mathsf{d}}

\DeclareMathOperator{\Dist}{dist}

\begin{document}

\title[semilinear parabolic equations  in metric spaces]
  {Existence of solutions semilinear parabolic equations with singular initial data in the Heisenberg group}

\authors

\author{The Anh Bui}
\address[The Anh Bui]
	{Department of Mathematics of Statistics, Macquarie University, NSW 2109,
	Australia.}
\email{the.bui@mq.edu.au}

	\author{Kotaro Hisa}
	\address[Kotaro Hisa]{Graduate School of Mathematical Sciences, The University of Tokyo, 3-8-1 Komaba,
Meguro-ku, Tokyo 153-8914, Japan.}
	\email{hisak@ms.u-tokyo.ac.jp}


	

\arraycolsep=1pt

\begin{abstract}
In this paper we obtain necessary conditions and sufficient conditions on the initial data for 
the solvability of fractional semilinear heat equations with power nonlinearities in the Heisenberg group $\mathbb{H}^N$.
Using these conditions, 
we can prove that  $1+2/Q$ separates the ranges of exponents of nonlinearities for the global-in-time solvability of the Cauchy problem (so-called the Fujita-exponent), where $Q=2N+2$ is the homogeneous dimension of $\mathbb{H}^N$,
and identify the optimal strength of the singularity of the initial data for the local-in-time solvability.
Furthermore, our conditions lead sharp estimates of the life span of  solutions with nonnegative initial data having a polynomial decay at the space infinity.
\end{abstract}

\maketitle

\tableofcontents

\section{Introduction.}

\subsection{Heisenberg group $\mathbb{H}^N$.}

This paper is concerned with nonnegative solutions of the fractional semilinear heat equation in the Heisenberg group 
\begin{equation}
\label{eq:Fujita}
\partial_t u +( -\Delta_{\mathbb{H}})^{\frac{\alpha}{2}} u =  u^p,\quad  \eta\in \mathbb{H}^N,  \,\,t>0,
\end{equation}
with the initial data
\begin{equation}
\label{eq:Fujitaini}
u(\cdot,0) =  \mu   \quad \mbox{in} \quad \mathbb{H}^N,
\end{equation}
%
where  $N\geq1$, $\alpha\in(0,2]$, $p>1$, and $\mu$ is a nonnegative Radon measure on $\mathbb{H}^{N}$.
Here, $\Delta_\mathbb{H}$ is the sub-Laplacian on the Heisenberg group $\mathbb{H}^N$ and 
 $(-\Delta_\mathbb{H})^{\alpha/2}$ denotes the fractional power of $-\Delta_\mathbb{H}$ on $\HH^N$.
For the exact definition of $(-\Delta_{\HH})^{\alpha/2}$,
see \eqref{eq:defofFSL} below.
 
In this paper we show that every nonnegative solution of \eqref{eq:Fujita} has a unique  nonnegative Radon
measure in $\mathbb{H}^N$ as the initial trace and study qualitative properties of it. Furthermore, we give
sufficient conditions for the solvability of  problem \eqref{eq:Fujita} with \eqref{eq:Fujitaini} and obtain sharp estimates
of the life span of  solutions with small initial data.
Throughout of this paper, for $d\ge1$ and $\alpha\in (0,2]$  denote
\[
p_{\alpha, d} := 1 + \frac{\alpha}{d}.
\]

The Heisenberg group is the Lie group $\mathbb{H}^N = \mathbb{R}^{2N+1}$ equipped with the group law
\[
\eta \circ \eta' = (x+x', y+y', \tau + \tau' + 2(x\cdot y'- x' \cdot y)),
\]
and with the Haar measure on $\mathbb{H}^N$, 
where $\eta = (x,y,\tau)$, $\eta'=(x',y',\tau') \in \mathbb{R}^{2N+1}$, and $\cdot$ is the scalar product in $\mathbb{R}^N$.
It is well-known that the Haar measure on $\mathbb{H}^N$ coincides with the $2N+1$-dimensional Lebesgue measure $\mathcal{L}^{2N+1}$.
The identity element for $\mathbb{H}^N$ is $0$ and $\eta^{-1} = - \eta$ for all $\eta\in \mathbb{H}^N$.
The homogeneous Heisenberg norm is defined by
\[
|\eta|_{\mathbb{H}^N} = ((|x|^2+|y|^2)^2 +\tau^2)^\frac{1}{4},
\]
where $|\cdot|$ is the Euclidean norm associated to $\mathbb{R}^N$.
Then 
\[
\dist_{\mathbb{H}} (\eta,\zeta) = |\zeta^{-1} \circ \eta|_{\mathbb{H}^N}
\]
is a left-invariant distance on $\mathbb{H}^N$.
The homogeneous dimension of $\mathbb{H}^N$ is $Q= 2N+2$.
The left-invariant vector fields that span the Lie algebra are given by
\[
X_i = \partial_{x_j} -2y_i \partial_\tau, \quad Y_i =\partial_{y_i} + 2x_i \partial_\tau,
\]
for $i=1,\cdots,N$.
The Heisenberg gradient is given by 
\[
\nabla_\mathbb{H} = (X_1,\cdots,X_N,Y_1,\cdots,Y_N),
\]
and the sub-Laplacian is defined as
\begin{equation}
	\label{eq:Lap}
\begin{split}
\Delta_\mathbb{H}
&  := \sum_{i=1}^N (X_i^2 + Y_i^2)\\
& = \Delta_x + \Delta_y +4(|x|^2+|y|^2)\partial_\tau^2 + 4\sum_{i=1}^N (x_i \partial^2_{y_i,\tau} - y_i \partial^2_{x_i,\tau}),
\end{split}
\end{equation}
where $\Delta_x$ and $\Delta_y$ stand for the Laplace operators on $\mathbb{R}^N$.
In particular, the Heisenberg group $\mathbb{H}^N$ is the most typical example of  metric measure spaces with a structure different from that of $\mathbb{R}^N$.

\subsection{Semilinear heat equations.} 

We recall the solvability of \eqref{eq:Fujita} in $\mathbb{R}^N$.
For $x\in\mathbb{R}^N$ and $r>0$, set $B_{\mathbb{R}^N}(x,r) := \{y\in\mathbb{R}^N: |x-y|<r\}$.
Let us consider nonnegative solutions of the fractional semilinear equation
\begin{equation}
\label{eq:fjt}
\partial_t v +( -\Delta)^{\frac{\alpha}{2}} v =  v^q,\quad  x\in \mathbb{R}^N,  \,\,t>0,
\end{equation}
with the initial data
\begin{equation}
\label{eq:fjtini}
v(\cdot,0) =  \nu,
\end{equation}
where $N\ge 1$, $\alpha \in (0,2]$, $q>1$, and $\nu$ is a nonnegative Radon measure on $\mathbb{R}^N$.
Here,  $(-\Delta)^{\alpha/2}$ denotes the fractional power of $-\Delta$ on $\mathbb{R}^N$.

Let us consider the case of $\alpha=2$ first.
The solvability of problem \eqref{eq:fjt} with \eqref{eq:fjtini}  has been studied in many papers since the pioneering work due to Fujita \cite{F} 
(see, e.g., \cite{QS19}, which includes an extensive list of references for \eqref{eq:fjt}).
It is already known from his work and the subsequent works of Hayakawa \cite{H73},  Sugitani \cite{S75}, and Kobayashi--Sirao--Tanaka \cite{KST77} 
that $1<q\le p_{2,N}$ implies the nonexistence of any positive global-in-time solutions, while  
if $q> p_{2,N}$, problem \eqref{eq:fjt} with \eqref{eq:fjtini} possesses a positive global-in-time solution for appropriate initial data $\nu$.
For this reason, such an exponent $p_{2,N}$ is called  the {\it  Fujita-exponent}.
Note that Sugitani \cite{S75} also dealt with the fractional case $\alpha\in(0,2)$.

Among these studies, Baras--Pierre \cite{BP85} obtained necessary conditions for the solvability in the case of $\alpha=2$.
Subsequently, the second author of this paper and Ishige \cite{HI18} obtained a generalization to the fractional case $\alpha \in (0,2)$ and sufficient conditions for the solvability.
Furthermore, they showed that every nonnegative solution of \eqref{eq:fjt} has a unique nonnegatuve Radon measure on $\mathbb{R}^N$ as the initial trace
(see also \cite{IKO20}, which deals with the case where $\alpha$ is a positive even integer as well as the case of $\alpha\in(0,2]$).
More precisely, the following results have already been obtained.

\begin{itemize}
\item[(a)] Let $N\ge1$, $\alpha\in (0,2]$, and $q>1$.
Let $v$ be a nonnegative solution of \eqref{eq:fjt} in $\mathbb{R}^N\times(0,T)$, where $T\in (0,\infty)$.
Then there exists a unique nonnegative Radon measure $\nu$ on $\mathbb{R}^N$ such that
\[
\operatorname*{ess\,lim}_{t\to0^+} \int_{\mathbb{R}^N} v(y,t) \phi(y) \, \di y = \int_{\mathbb{R}^N} \phi(y) \, \di \nu(y)
\]
for all $\phi\in C_c(\mathbb{R}^N)$.
\item[(b)]  Let $\nu$ be as in assertion (a). Then $\nu$ must satisfy the following:
\begin{itemize}
\item If $1<q<p_{\alpha,N}$, then $\displaystyle{\sup_{x\in\mathbb{R}^N} \nu(B_{\mathbb{R}^N}(x,T^\frac{1}{\alpha}))\le\gamma T^{\frac{N}{\alpha}-\frac{1}{q-1}}}$;
\item If $q=p_{\alpha,N}$, then $\displaystyle{\sup_{x\in\mathbb{R}^N} \nu(B_{\mathbb{R}^N}(x,\sigma))\le\gamma \left[\log\left(e+\frac{T^{1/\alpha}}{\sigma}\right)\right]^{-{\frac{N}{\alpha}}}}$ for all $0<\sigma<T^\frac{1}{\alpha}$;
\item If $q>p_{\alpha,N}$, then $\displaystyle{\sup_{x\in\mathbb{R}^N} \nu(B_{\mathbb{R}^N}(x,\sigma))\le\gamma\sigma^{N-\frac{\alpha}{q-1}}}$ for all $0<\sigma<T^\frac{1}{\alpha}$.
\end{itemize}
Here $\gamma>0$ is a constant depending only on $N$, $\alpha$, and $q$.
\end{itemize}
From this necessary condition (b) we can find a large constant $C_*>0$ with the following property:
\begin{itemize}
\item[(c)] Problem \eqref{eq:fjt} with \eqref{eq:fjtini} possesses no local-in-time solutions if $\nu$ is a nonnegative measurable function in $\mathbb{R}^N$ satisfying
$\nu(x)\ge C_*\Phi(x)$
in a neighborhood of the origin, where 
\begin{equation*}
\Phi(x) := 
\left\{
\begin{array}{ll}
\displaystyle{|x|^{-N}\left[\log\left(e+\frac{1}{|x|}\right)\right]^{-\frac{N}{\alpha}-1}} & \mbox{if} \quad q=p_{\alpha,N},\vspace{3pt}\\
\displaystyle{|x|^{-\frac{\alpha}{q-1}}} \quad & \mbox{if} \quad q>p_{\alpha,N}.\vspace{3pt}\\
\end{array}
\right.
\end{equation*}
\end{itemize}
On the other hand, the following sufficient condition for the  solvability  has also been obtained:
\begin{itemize}
\item[(d)] There exists a small constant $c_*>0$ 
such that
if $\nu$ satisfies
$0\le\nu(x)\le c_*\Phi(x)$ in $\mathbb{R}^N$,
then problem~\eqref{eq:fjt} with \eqref{eq:fjtini} possesses a local-in-time solution.
In particular, when $p>p_{\alpha,N}$, this solution is a global-in-time one.
\end{itemize}
Note that these results show that $q=p_{\alpha,N}$ is the Fujita-exponent.
For (d), see also the papers written by Kozono--Yamazaki \cite{KY94} ,
Robinson--Sier\.{z}\k{e}ga \cite{RS13} (the case where $\alpha =2$ and $q>p_{2,N}$),
and
Ishige--Kawakami--Okabe \cite{IKO20} (the case where $\alpha\in(0,2]$ and $q\ge p_{\alpha,N}$).
The conditions (c) and (d) demonstrate that strength of the singularity at the origin of the function $\Phi$ is the critical threshold for the local-in-time solvability of problem \eqref{eq:fjt} with \eqref{eq:fjtini}.
We term such a singularity in the initial data an {\it optimal singularity} for the solvability of  problem \eqref{eq:fjt} with \eqref{eq:fjtini}.
For studies on optimal singularities, see e.g. \cites{FHIL23, FHIL24, H24,HI18, HIT23, HS24, HT21, IKO20}.

\vspace{5pt}
Let us go back to \eqref{eq:Fujita}.
For problem \eqref{eq:Fujita} with \eqref{eq:Fujitaini} in the case of $\alpha=2$, Zhang \cite{Z98} proved that $1<p<p_{2,Q}$ implies the nonexistence of any positive global-in-time solutions, while  
if $p>p_{2,Q}$, problem \eqref{eq:Fujita} with \eqref{eq:Fujitaini} possesses a positive global-in-time solution for appropriate initial data $\mu$. Later, Pohozaev--V\'eron \cite{PV00} considered more general equations in $\mathbb{H}^N$ and proved that $p= p_{2,Q}$ is included in the nonexistence case. 
Namely, we already know that $p=p_{2,Q}$ is the Fujita exponent.
Recently,   Georgiev--Palmieri \cite{GP21} obtained sharp estimates of the life span of solutions of problem \eqref{eq:Fujita} with small initial data in the case of $1<p\le p_{2,Q}$
and also proved the global-in-time solvability in the case of $p>p_{2,Q}$.
Many other studies on the global-in-time solvability in $\mathbb{H}^N$ have been undertaken, see e.g. \cites{A01, AJS15, FRT24, JKS16,P98, P99, RY22}.
We would like to emphasize that,  in the case of $\alpha\in(0,2)$, there were no previous studies on
 problem \eqref{eq:Fujita} with \eqref{eq:Fujitaini}  to our knowledge.


\vspace{5pt}
In this paper we consider \eqref{eq:Fujita} and  obtain analogous results to \cite{HI18} (i.e. assertions (a)--(d)) in the Heisenberg group $\mathbb{H}^N$.
Our conditions are optimal and, as an application, we show that our conditions can lead the sharp estimates of the life span of solutions with $\mu$ having  a polynomial decay at the space infinity.

\subsection{Notation and the definition of solutions.}\label{Semilinear heat equations}

In order to state our main results, we 
prepare some notation and formulate the definition of solutions.
For any measurable set $A\subset \mathbb{H}^{N}$, $|A|$ denotes the Haar measure of $A$.
For any $\eta \in \mathbb{H}^N$ and $r>0$, set $B(\eta,r):= \{\zeta \in \mathbb{H}^N: \dist_{\mathbb{H}} (\eta, \zeta) <r\}$. 
Furthermore, for any $L^1_{\rm loc}$ function $f$, we set 
\[
 \fint_{B(\eta,r)} f(\zeta) \, \di \zeta := \frac{1}{|B(\eta,r)|} \int_{B(\eta,r)} f(\zeta) \, \di \zeta.
\] 
Following \cite{Folland}, for $\alpha\in (0,2)$ we define the fractional sub-Laplacian by
\begin{equation}
\label{eq:defofFSL}
(-\Delta_{\HH})^{\frac{\alpha}{2}} f := \frac{1}{\Gamma(1 - \alpha/2)} \int_{0}^{\infty} t^{-\frac{\alpha}{2}} (-\Delta_{\HH}) e^{t\Delta_\HH} f \, \di t,
\end{equation}
where $\Gamma$ is the Gamma function, $e^{t\Delta_\mathbb{H}}$ is the heat flow, and $f \in L^2(\mathbb{H}^N)$ is any function for which the relevant limit exists in $ L^2$ norm.
For the simplicity of notation, for $\alpha\in(0,2]$, denote
\[
\Lambda_\alpha := -(-\Delta_\mathbb{H})^\frac{\alpha}{2}.
\]
Let $G_\alpha:=G_\alpha(\eta, t)$ be the fundamental solution of 
\begin{equation}
\label{eq:LHE}
\partial_t u- \Lambda_\alpha u = 0 \quad \mbox{in} \quad \mathbb{H}^N\times (0,\infty),
\end{equation}
where $\alpha \in (0,2]$.
When $\alpha=2$, we briefly write $G(\eta,t)$ instead of $G_\alpha(\eta,t)$. For any locally integrable function $f$ in $\mathbb{H}^N$, it is well-known that
\[
[e^{-t(-\Delta_\mathbb{H})^{\alpha/2}}f](\eta) = \int_{\HH^N}G_\alpha( \zeta^{-1}\circ \eta,t) f(\zeta) \, \di \zeta.
\]
See e.g.,  \cite{FS} (for $\alpha=2$) and  \cite{MPS} (for $\alpha \in (0,2)$).

The key to our arguments is the following two sided estimate of the fundamental solution $G_\alpha$ of \eqref{eq:LHE}.
\begin{proposition}
	\label{lem - Gaussian implies A1 A2}
		 $G_\alpha(\eta,t)$ satisfies the following estimate: there exist $C_1, C_2, c_1, c_2>0$ such that \begin{equation}\label{A1}
			\f{C_1 }{t^{Q/\alpha}}g_\alpha\left(\f{|\eta|_{\mathbb H^N}}{c_1t^{1/\alpha}}\right)\le G_\alpha(\eta,t)\le \f{C_2 }{t^{Q/\alpha}}g_\alpha\left(\f{|\eta|_{\mathbb H^N}}{c_2t^{1/\alpha}}\right)
		\end{equation}
		for  $\eta \in \HH$ and $t>0$, where  
		\[
		g_\alpha(s)=
		 \left\{
	\begin{aligned}
			&(1 +s)^{-Q-\alpha} \quad 
			&&\mbox{if} \quad \alpha \in (0,2),\\
			&e^{-s^2} \quad 
			&& \mbox{if} \quad  \alpha=2.
	\end{aligned}
	\right.
		\] 	
\end{proposition}
For the proof of Proposition~\ref{lem - Gaussian implies A1 A2}, see Subsection \ref{Some kernel estimates} below.
In the case of $\alpha =2$, \eqref{A1} is well-known (see e.g. \cite{FS}). 
But in the case of $\alpha\in(0,2)$, this is new. To prove the estimates \eqref{A1} we employ a subordination formula in \cite{G}, which allows us to estimate the kernel $G_\alpha(\eta,t)$ via the heat kernel of the sub-Laplacian $-\Delta_{\mathbb H}$.

%

We formulate the definition of solutions.

\begin{definition}
Let $u$ be a nonnegative measurable function in $\HH^N \times (0,T)$, where $T\in (0,\vc]$. 
\begin{itemize}
\item[(i)] We say that $u$ is a solution of \eqref{eq:Fujita} in $\HH^N \times (0,T)$ if $u$ satisfies
\begin{equation}
\label{eq - additional condition solution}
\infty > u(\eta,t) =
[e^{(t-\tau)\Lambda_\alpha}u(\tau)](\eta)  
+\int_{\tau}^{t} [e^{(t-s)\Lambda_\alpha}u(s)^p](\eta) \, \di s
\end{equation}
for almost all (a.a.)  $\eta\in \HH^N$ and $0 < \tau < t < T$.
\item[(ii)] We say that $u$ is a solution of problem \eqref{eq:Fujita} with \eqref{eq:Fujitaini} in $\HH^N \times [0,T)$ if $u$ satisfies
\begin{equation}
\label{eq- defn sols of parabolic eqs}
\infty > u(\eta,t) =[e^{t\Lambda_\alpha}\mu](\eta) + \int_{0}^{t} [e^{(t-s)\Lambda_\alpha}u(s)^p](\eta) \, \di s 
\end{equation}
for a.a.~$\eta\in \HH^N$ and $0 < t < T$. If $u$ satisfies \eqref{eq- defn sols of parabolic eqs}  with $=$ replaced by $\geq$, then $u$ is said to be a supersolution of problem \eqref{eq:Fujita} with \eqref{eq:Fujitaini}  in $\HH^N\times[0,T)$.
\item[(iii)] Let $u$ be a solution of problem \eqref{eq:Fujita} with \eqref{eq:Fujitaini} in $\HH^N\times [0,T)$. We say that $u$ is a minimal solution  if
		\[
		u(\eta,t) \leq v(\eta,t) \quad \text{for a.a.}  \quad \eta\in \HH^N   \,\, \mbox{and} \,\, 0 < t < T
		\]
		for any solution $v$ of problem \eqref{eq:Fujita} with \eqref{eq:Fujitaini} in $\HH^N\times[0,T)$. 
	\end{itemize}
\end{definition}

\subsection{Main results.}

Now we are ready to state the main results of this paper. In the first theorem we show the existence and the uniqueness of the initial trace of  solutions of \eqref{eq:Fujita} and obtain  analogous results to assertions (a) and (b).

\begin{thm}
\label{main thm}
Let  $N\ge1$, $\alpha\in(0,2]$, and $p > 1$. Let $u$ be a solution of \eqref{eq:Fujita} in $\HH^N\times(0,T)$, where $T\in(0, \infty)$. Then there exists a unique Radon measure $\mu$ on $\HH^N$ such that
\begin{equation}
\label{eq- mainthm 1}
\operatorname*{ess\,lim}_{t \to 0^+} \int_{\HH^N} u(\zeta,t) \phi(\zeta) \, \di\zeta  = \int_{\HH^N} \phi(\zeta)   \, \di\mu(\zeta) 
\end{equation}
for all $\phi \in C_c(\HH^N)$.
 Furthermore, there exists $\gamma_A > 0$ depending only on $Q$, $\alpha$, and $p$ such that 
 \begin{itemize}
\item[(i)]
$\displaystyle{\sup_{\eta\in \HH^N} \mu(B(\eta, T^\frac{1}{\alpha})) \le \gamma_A T^{\frac{Q}{\alpha} - \frac{1}{p-1}}}$
\quad if $1<p<p_{\alpha,Q}$;
\item[(ii)]
$\displaystyle{\sup_{\eta\in \HH^N} \mu(B(\eta, \sigma)) \le \gamma_A \left[ \log \left( e + \frac{T^{{1/\alpha}}}{\sigma} \right) \right]^{-\f{Q}{\alpha}}}$
\quad for all $0<\sigma < T^\f{1}{\alpha}$ if $p=p_{\alpha,Q}$;
\item[(iii)]
$\displaystyle{\sup_{\eta\in \HH^N} \mu(B(\eta, \sigma)) \le \gamma_A \sigma^{Q - \frac{\alpha}{p-1}}}$
\quad for all $0<\sigma < T^\f{1}{\alpha}$ if $p>p_{\alpha,Q}$.
\end{itemize}
\end{thm}

Since $\mu$ in \eqref{eq- mainthm 1} is unique, the following holds:
\begin{remark}
If problem \eqref{eq:Fujita} with \eqref{eq:Fujitaini} possesses a  solution  in $\mathbb{H}^N\times[0,T)$, where $T\in(0,\infty)$, then $\mu$ satisfies assertions {\rm (i)--(iii)} in Theorem {\rm \ref{main thm}}.
\end{remark}
\begin{remark}
Since $Q/\alpha - 1/(p-1)<0$ when $1<p<p_{\alpha,Q}$, assertion (i) in Theorem \ref{main thm} is equivalent to
\[
\sup_{\eta\in \HH^N} \mu(B(\eta, \sigma)) \le \gamma_A \sigma^{Q - \frac{\alpha}{p-1}}
\]
for all $0<\sigma < T^{1/\alpha}$.
\end{remark}

\begin{remark}
\label{Remark:1.2}
 Let $1<p\le p_{\alpha,Q}$ and $u$ be a solution of problem \eqref{eq:Fujita} with \eqref{eq:Fujitaini} in $\mathbb{H}^N\times[0,\infty)$.
It follows from the assertions {\rm (i)} and {\rm (ii)} in Theorem {\rm \ref{main thm}} that $\mu$ must be zero in $\mathbb{H}^N$.
Then, in the case of $\alpha=2$,  Theorem {\rm \ref{main thm}} leads the same conclusion as \cites{A01, AJS15, FRT24, JKS16,P98, P99,PV00, RY22,Z98}.
\end{remark}

As a corollary of Theorem \ref{main thm}, we have:
\begin{corollary}\label{cor 1}
	Let  $N\ge1$, $\alpha\in (0,2]$, and $p > 1$. Let $u$ be a solution of problem \eqref{eq:Fujita} with \eqref{eq:Fujitaini} in $\HH^N\times[0,T)$, where $T\in(0, \infty)$. Then there exists $\gamma'_A> 0$ depending only on $N$, $\alpha$, and $p$ such that
	\begin{equation}\label{eq1-cor1}
	\sup_{\eta\in \HH^N} \fint_{B(\eta,(T-t)^{\frac{1}{\alpha}})} u(\zeta,t) \, \di\zeta  \leq \gamma'_A (T - t)^{-\frac{1}{p-1}}
	\end{equation}
	for a.a.~$0 < t < T$.
	
\end{corollary}

Theorem \ref{main thm} can be regarded as a generalization of the result in \cite[Theorem 1.1]{HI18}.
Let $u$ be a solution of problem \eqref{eq:Fujita} with \eqref{eq:Fujitaini} in $\mathbb{H}^N\times[0,T)$, where $T\in (0,\infty)$. 
We first prove the existence and the uniqueness of the initial trace of the solution $u$
and then obtain necessary conditions for the solvability (i.e. assertions (i)--(iii) in Theorem \ref{main thm}).
The proof of our necessary conditions in the case of $\alpha\in (0,2)$ follows the arguments in the proofs of \cites{FHIL23, H24, HIT23, HS24, LS21}, which enable us to more easily obtain the necessary conditions in \cite[Theorem 1.1]{HI18}.
Let $\zeta\in\mathbb{H}^N$.
Following these results, we get an integral inequality related to 
\begin{equation}
\label{eq:FHIL}
t^\frac{Q}{\alpha} \int_{\mathbb{H}^N} G_\alpha(\zeta^{-1}\circ\eta,t) u(\eta,t) \, \di \eta,
\end{equation}
and then obtain the desired estimates by applying the existence theorem for ordinary differential equations to this estimate.
See Lemma~\ref{Lemma:XODE} below.

However, the proof in the case of $\alpha=2$ is completely different from those in \cites{FHIL23, H24, HIT23, HS24, LS21} and \cite{HI18}.
First, we extend solutions to the framework of weak ones
and then employ a suitable cut-off function as a test function.
This method was developed by Mitidieri--Pohozaev \cite{MP01} and can be applied not only to semilinear heat equations but also to a wider class of equations.
For the case of semilinear heat equations, see e.g. \cites{GP21, IKO20, IS19}.
In this paper we follow the arguments in \cite{IKO20}.

The reason why we adopt  two methods  comes from the fact that in Proposition~\ref{lem - Gaussian implies A1 A2}, $c_1=c_2$ does generally not hold in the case of $\alpha=2$.
In order to apply the arguments in \cites{FHIL23, H24, HIT23, HS24, LS21}, we have to get 
\[
G_\alpha(\zeta^{-1}\circ\eta, 2t-s) \ge \left(\frac{s}{2t}\right)^\frac{Q}{\alpha} G_\alpha(\zeta^{-1}\circ\eta, s)
\]
for $\eta,\zeta\in\mathbb{H}^N$ and $0<s<t$.
However, in the case of $\alpha =2$,  we can not get such an estimate from Proposition~\ref{lem - Gaussian implies A1 A2}, since $g_\alpha$ is an exponential function.
On the other hand, in the case of $\alpha \in (0,2)$,  it can be regarded as $c_1=c_2=1$ and we can avoid this problem, since $g_\alpha$ is a polynomial function.
In the previous studies dealing with $\mathbb{R}^N$ \cites{H24, HI18, HIT23, HS24,LS21,FHIL23}, this problem did not arise because they were either dealing with the fractional case  or had an explicit formula of $G_\alpha$.

By Theorem \ref{main thm} we have
\begin{thm}
	\label{mainthm 2}
Let $u$ be a solution of \eqref{eq:Fujita} in $\HH^N\times(0,T)$, where $T\in (0,\infty)$. If $\mu$ is a nonnegative Radon measure  satisfying \eqref{eq- mainthm 1}, then  u is a solution of problem \eqref{eq:Fujita} with \eqref{eq:Fujitaini} in $\HH^N\times[0,T)$.
\end{thm}

We give sufficient conditions for the solvability of problem \eqref{eq:Fujita} with \eqref{eq:Fujitaini}.
The proofs follow the arguments in \cite{HI18}.
\begin{thm}
	\label{mainthm 3}
 Let  $N\ge1$, $\alpha\in(0,2]$, and $1 < p < p_{\alpha,Q}$. Then there exists $\gamma_C > 0$ such that, if $\mu$ is a nonnegative Radon mesure on $\HH^N$ satisfying
\[
\sup_{\eta\in \HH^N} \mu(B(\eta, T^{\frac{1}{\alpha}})) \le \gamma_C T^{\frac{Q}{\alpha} - \frac{1}{p-1}} \quad \text{for some } T > 0, 
\]
then problem \eqref{eq:Fujita} with \eqref{eq:Fujitaini} possesses a solution in $\HH^N\times[0,T)$.
\end{thm}

\begin{thm}
	\label{mainthm 4}
 Let $N\ge1$, $\alpha\in(0,2]$, and $\theta>1$. Then there exists $\gamma_D > 0$ such that, if $\mu$ is a nonnegative measurable function in $\HH^N$ satisfying
\begin{equation}
\label{eq:SC2}
\sup_{\zeta\in \HH^N} \left[ \fint_{B(\zeta,\sigma)} \mu(\eta)^\theta \, \di\eta  \right]^{\frac{1}{\theta}} \leq \gamma_D \sigma^{-\frac{\alpha}{p-1}}, \quad 0 < \sigma < T^{\frac{1}{\alpha}} 
\end{equation}
for some $T > 0$, then problem \eqref{eq:Fujita} with \eqref{eq:Fujitaini} possesses a solution in $\HH^N\times[0,T)$.
\end{thm}

\begin{thm}
	\label{Theorem:SC3}
	Let $N\ge1$, $\alpha\in(0,2]$, $p=p_{\alpha,Q}$, and $\beta>0$. 
	For $s>0$, set
	\begin{equation}
		\label{eq:SC31}
		\Psi_\beta(s) := s[\log(e+s)]^\beta, \qquad
		\rho(s) := s^{-Q} \left[\log\left(e+\frac{1}{s}\right)\right]^{-\frac{Q}{\alpha}}.
	\end{equation}
	Then there exists $\gamma_E > 0$ such that, if $\mu$ is a nonnegative measurable function in $\HH^N$ satisfying
	\begin{equation}
		\label{eq:SC32}
		\sup_{\zeta \in \mathbb{H}^N}  \Psi_\beta^{-1}\left[ \fint_{B(\zeta,\sigma)} \Psi_\beta(T^\frac{1}{p-1} \mu(\eta)) \, \di\eta \right] \le \gamma_E \rho(\sigma T^{-\frac{1}{\alpha}}), \quad 0 < \sigma < T^\frac{1}{\alpha} 
	\end{equation}
	for some $T > 0$, then problem \eqref{eq:Fujita} with \eqref{eq:Fujitaini} possesses a solution in $\mathbb{H}^N\times[0,T)$.
\end{thm}
As a corollary of Theorems \ref{main thm}, \ref{mainthm 3}, \ref{mainthm 4}, and \ref{Theorem:SC3}, we have
\begin{corollary}
\label{Cor:1.2}
Let $N\ge1$, $\alpha\in(0,2]$, and $p\ge p_{\alpha,Q}$.
Define
\begin{equation*}
\Phi_\alpha(\eta) := 
\left\{
\begin{array}{ll}
\displaystyle{|\eta|_{\mathbb{H}^N}^{-Q}\left[\log\left(e+\frac{1}{|\eta|_{\mathbb{H}^N}}\right)\right]^{-\frac{Q}{\alpha}-1}} & \mbox{if} \quad p=p_{\alpha,Q},\vspace{3pt}\\
\displaystyle{|\eta|_{\mathbb{H}^N}^{-\frac{\alpha}{p-1}}} \quad & \mbox{if} \quad p>p_{\alpha,Q}.\vspace{3pt}\\
\end{array}
\right.
\end{equation*}
Assume that $\mu(\eta) = \gamma \Phi_\alpha(\eta) +C_\alpha$ in $\mathbb{H}^N$ for some $\gamma\ge0$ and $C_\alpha\ge0$.
Then there exists $\gamma_*>0$ with the following properties: 
\begin{itemize}
\item[(1)] problem \eqref{eq:Fujita} with \eqref{eq:Fujitaini} possesses a local-in-time solution if $0\le\gamma<\gamma_*$;
\item[(2)] problem \eqref{eq:Fujita} with \eqref{eq:Fujitaini} possesses no local-in-time solutions if $\gamma>\gamma_*$.
\end{itemize}
In particular, if $p>p_{\alpha,Q}$ and  $C_\alpha=0$, the solution in the assertion {\rm (1)} is a global-in-time one. 
\end{corollary}
From this corollary it can be seen that $\Phi_\alpha$ is an optimal singularity for the local-in-time solvability.
In particular, together with Remark \ref{Remark:1.2}, it can be seen that $p_{\alpha,Q} = 1+ \alpha/Q$ is the Fujita-exponent.

The rest of this paper is organized as follows. 
In Section~2 we collect properties of $\mathbb{H}^N$ and $G_\alpha$ and prepare some preliminary lemmas. 
In Section~3 we prove \eqref{eq- mainthm 1} in Theorem~\ref{main thm} and Theorem~\ref{mainthm 2}.
In Section~4 we prove assertions (i)--(iii) in Theorem~\ref{main thm} and  complete the proof of Theorem~\ref{main thm}.
In Section~5 we prove Theorems~\ref{mainthm 3}--\ref{Theorem:SC3}.
In Section~6, as an application of our theorems, we obtain estimates of the life span of solutions of problem \eqref{eq:Fujita} with  small initial data.

\section{Preliminaries.}
In what follows, the letters $C$ and $C'$ denote  generic positive constants depending only on $N$, $\alpha$, and $p$.
For any two nonnegative functions $f_1$ and $f_2$ defined on a subset $D\subset \mathbb{R}$, 
we write $f_1(\tau) \lesi  f_2(\tau)$ for all $\tau \in D$ if 
$f_1(\tau) \le C f_2(\tau)$ for all $\tau \in D$,
and
we write $f_1(\tau) \sim f_2(\tau)$ for all $\tau \in D$ if 
$ f_1(\tau) \lesi  f_2(\tau)$ and $ f_2(\tau) \lesi  f_1(\tau)$ for all $\tau \in D$.
Furthermore, for $A,B \ge0$, we write  $A\simeq B$ if  $C B \le A \le C' B$ for some constants $0<C<C'$.
\subsection{Basic properties of $\mathbb{H}^N$ and  $G_\alpha$.}

In this subsection we collect properties of the Heisenberg group
$\mathbb{H}^N$ and the fundamental solution $G_\alpha$.
The following lemma is used when we calculate  integrals in the Heisenberg group $\mathbb{H}^N$ 
and is well-known (see e.g. \cites{BHQ24, P98}). Therefore, we omit the proof.
\begin{lemma}
Let $N\ge1$,  $\zeta\in\mathbb{H}^N$, $a>0$,  $f:[0,\infty)\to[0,\infty)$ be a continuous function. Then one has
\begin{equation}
\label{eq:6.1}
\int_{B(\zeta,a)} f(| \zeta^{-1}\circ\eta|_{\mathbb{H}^N}) \, \di \eta 
\simeq \int_{0}^{a^2} f(\sqrt{r}) r^N \, \di r.
\end{equation}
Moreover,  one has
\begin{equation}
\label{eq:6.2}
\int_{B(\zeta,a)}  \, \di \eta  = |B(\zeta,a)| = |B(0,a)| \simeq a^Q.
\end{equation}
\end{lemma}

For $\lambda>0$ and $\eta = (x,y,\tau)\in \mathbb{H}^N$, define
 	\[
 	\delta_\lambda(\eta) := (\lambda x, \lambda y , \lambda^2\tau ).
 	\]
The following properties are taken from \cites{FS, MPS}.
Given $\alpha \in (0, 2]$, the function $ G_\alpha$ has the following properties:

\begin{align}
	&\notag G_\alpha \in C^{\infty}(\HH^N\times(0,\infty)),\\
	&\label{eq:G2}G_\alpha(\eta,t) = G_\alpha(\eta^{-1},t),\\	
	&\label{eq:G3}G_\alpha(\delta_{\lambda}(\eta),\lambda^{\alpha} t) = \lambda^{-Q} G_\alpha(\eta,t),\\
	& \label{eq:G4}G_\alpha(\eta,t) = \int_{\mathbb{H}^N}  G_\alpha(\zeta^{-1}\circ \eta, t-s) G_\alpha(\zeta,s) \, \di \zeta,\\
	&\label{eq:G5}\displaystyle{\int_{\HH^N} G_\alpha(\eta,t) \, \di \eta = 1},
\end{align}
for all $\eta\in \mathbb{H}^N$, $0<s<t$, and $\lambda>0$.

Since $X^*_i = - X_i$, $Y_i^*=-Y_i$ for $i=1,\ldots, N$ and $\Delta_\HH = \sum_{i=1}^N (X_i^2 +Y_i^2)$, by integration by parts we have: 
	\begin{equation}
		\label{eq:IBP}
		\int_{\mathbb{H}^N} -\Delta_\mathbb{H} \phi(\eta) \cdot \psi(\eta) \, \di \eta
		= \int_{\mathbb{H}^N} \phi(\eta) ( -\Delta_\mathbb{H})\psi(\eta) \, \di \eta
	\end{equation}
	for all $\phi, \psi \in C_0^\infty(\mathbb{H}^N)$.

\subsection{A covering lemma and the Hardy-Littlewood maximal function.}
For $\eta\in\mathbb{H}^N$ and a nonempty subset $A\subset \mathbb{H}^N$,
set
\[
\Dist(\eta,A) := \inf\{\dist_\mathbb{H}(\eta,\overline{\eta}): \overline{\eta}\in A^c\}.
\]

\begin{lemma}\label{lem-covering lemma}
\begin{enumerate}
\item[\rm (a)] Let $0<r<R<\vc$ and $\eta\in \HH^N$. Then there exists an universal constant $C$ such that we can find a family of balls $\{B_k:=B(\eta_k,r): \eta_k\in B(\eta,R), k\in I\}$ for some countable family of indices $I$ such that
		\begin{itemize}
			\item[(i)] $\displaystyle{B(\eta,R)\subset \bigcup_{k\in I}B_k}$;
			\item[(ii)] $\displaystyle \sharp I \le C\left(R/r\right)^Q$.
		\end{itemize} 
\item[\rm (b)] Let  $R>0$ and $\eta_0\in \mathbb{H}^N$. Then for each $k=2,3,\ldots$, we can find a family of balls $\{B^k_j:=B(\eta^k_j, R): j\in J\}$ for some countable family of indices $J$ such that 
		\begin{itemize}
			\item[(i)] $\displaystyle{B(\eta_0,2^{k+1}R)\backslash B(\eta_0,2^{k}R)\subset \bigcup_{j\in J}B^k_j}$;
			\item[(ii)] ${\rm dist}(\eta_0,B^k_j)\simeq 2^{k}R$ for each $j\in J$ and $k\ge 2$;
 			\item[(iii)] $\displaystyle \sharp J \le C2^{kQ}$, where $C$ is a constant independent of $k, R$, and $\eta_0$.
		\end{itemize}
	\end{enumerate}
Here, for any set $I$, $\sharp I$ denotes the cardinal number of $I$.
\end{lemma}
\begin{proof}
	Since the proof of (a) is similar to that of (b) and even easier, we need only to prove (b).
	Fix $k \ge 2$. We consider the following family of balls $\{B(\eta,R/5): \eta
	\in B(\eta_0,2^{k+1}R)\backslash B(\eta_0,2^{k}R)\}$, which covers $B(\eta_0,2^{k+1}R)\backslash B(\eta_0,2^{k}R)$. By Vitali's covering lemma, we can extract a disjoint family of balls $\{ B(\eta^k_j, R/5):j\in J\}$ for some countable family of indices $J$ satisfying \[
	B(\eta_0,2^{k+1}R)\backslash B(\eta_0,2^{k}R)\subset \bigcup_{j\in J}B(\eta^k_j,R).
	\]
	By the construction, it is straightforward that the family $\{B^k_j:=B(\eta^k_j, R):j\in J\}$ satisfies (i) and (ii). In addition, by \eqref{eq:6.2},
	\[
	\begin{aligned}
		(2^kR)^Q&\gtrsim  |B(\eta_0,2^{k+1}R+R/5)\backslash B(\eta_0,2^{k}R-R/5)| \ge 
		\sum_{j\in J}|B(\eta_j^k,R/5)|\gtrsim R^Q \times \sharp J,
	\end{aligned}
	\]
	which implies (iii).
	This competes our proof.
\end{proof}
Recall that the Hardy-Littlewood maximal function $\mathcal M$ is defined by 
\[
\mathcal M f(\eta) =\sup_{B\ni \eta} \f{1}{r_B^Q}\int_{B} f(\zeta)\di \zeta,
\]
where the supremum is taken over all balls $B$ containing $\eta$ and $r_B>0$ is the radius of $B$. It is well-known that $\mathcal M$ is bounded on $L^p(\HH^N)$ for $1<p\le \vc$.

We have the following result whose proof is quite elementary and will be omitted.
\begin{lemma}\label{lem-maximal function}
	For $\epsilon>0$, there exists $C>0$ such that
	\[
	\int_{\HH^N} \f{1}{t^{Q/\alpha}}\left(1+\f{\dist_\mathbb{H}(\eta,\zeta)}{t^{1/\alpha}}\right)^{-(Q+\epsilon)}|f(\zeta)|\di \zeta \le C\mathcal M f(\eta)
	\]
	for all $\eta\in \HH^N$, $t>0$ and $f\in L^1_{\rm loc}(\mathbb H^N)$.
\end{lemma}
\subsection{Some kernel estimates.} \label{Some kernel estimates}

In this subsection, we obtain estimates of the fundamental solution $G_\alpha$ and collect the basic properties of $\Lambda_\alpha$.

 First, we prove Proposition~\ref{lem - Gaussian implies A1 A2} in Subsection~\ref{Semilinear heat equations}.

\begin{proof}[Proof of Proposition~{\rm \ref{lem - Gaussian implies A1 A2}}]
	The case $\alpha=2$ is well-known. In fact, we have 
	\begin{equation}
		\label{eq- Gaussian of ptxy}
		\f{C_1}{t^{Q/2}}\exp\left(-\f{|\eta|_{\HH^N}^2}{c_1t}\right)\le G(\eta,t)\le \f{C_2}{t^{Q/2}}\exp\left(-\f{|\eta|_{\HH^N}^2}{c_2t}\right) 
	\end{equation}
	for all $\eta\in \HH^N$ and $t>0$. See e.g. \cite{FS}.
	
	It remains to prove for the case $\alpha\in (0,2)$.	
	We use the following subordination formula in \cite{G}:
	\begin{equation}\label{eq-sub formula}
	e^{t\Lambda_\alpha}=\int_0^\vc e^{s \Lambda_2 }\phi_t^\alpha(s)\, \di s,
	\end{equation}
	where the function $\phi^\alpha_t:[0,\infty) \to [0,\infty)$  satisfies the following properties:
	\begin{enumerate}[{\rm (i)}]
		\item  $\phi^\alpha_t(s)\ge 0$ and $\displaystyle{\int_0^\vc\phi_t^\alpha(s) \,\di s =1}$;
		\item $\phi_t^\alpha(s)=t^{-\f{2}{\alpha}}\phi_1^\alpha(st^{-\f{2}{\alpha}})$;
		\item $\displaystyle{\phi_t^\alpha(s)\le  \f{Ct}{s^{1+\alpha/2}}}$ for all $s,t>0$;
		\item $\displaystyle{\phi_t^\alpha(s)\sim \f{t}{s^{1+\alpha/2}}}$ for all $s\geqslant t^{\f{2}{\alpha}}>0$;
		\item $\displaystyle{\int_0^\vc s^{-\gamma}\phi_1^\alpha(s) \, \di s<\vc}$ for all $\gamma>0$.
	\end{enumerate}
	We first establish an upper bound for $G_\alpha(\eta,t)$.  By \eqref{eq- Gaussian of ptxy}, \eqref{eq-sub formula}, (iii), and (iv), we have
	\[
	\begin{aligned}
		G_\alpha(\eta,t)\lesi \int_0^{\vc}\f{t}{s^{(Q+\alpha)/2}}\exp\Big(-\f{|\eta|_{\HH^N}^2}{c_2s}\Big)  \f{\di s}{s}.
	\end{aligned}
	\]
	\textbf{Case 1: $| \eta|_{\HH^N}\ge t^{1/\alpha}$.} Using the change of variable $u= |\eta|_{\HH^N}^{2}/s$,
	\[
	\begin{aligned}
		G_\alpha(\eta,t)&\lesi \int_0^{\vc}\f{t u^{{(Q+\alpha)/2}}}{| \eta|_{\HH^N}^{Q+\alpha}}\exp\left(-\f{u}{c_2}\right)  \f{\di u}{u}\\
		&\lesi \f{t}{|\eta|_{\HH^N}^{Q+\alpha}} \simeq  \f{1}{t^{Q/\alpha}}\left(1+\f{| \eta|_{\HH^N}}{t^{1/\alpha}}\right)^{-(Q+\alpha)}.
	\end{aligned}
	\]
	\textbf{Case 2: $|\eta|_{\HH^N}< t^{1/\alpha}$.} In this case, by \eqref{eq-sub formula},   \eqref{eq- Gaussian of ptxy}, and (ii),
	\[
	\begin{aligned}
				G_\alpha(\eta,t)&\lesi \int_0^{\vc} \f{1}{s^{Q/2}}t^{-\frac{2}{\alpha}}\phi_1^\alpha(st^{-\frac{2}{\alpha}}) \, \di s\\
				&\simeq \int_0^{\vc} \f{t^{-{Q/\alpha}}}{u^{Q/2}} \phi_1^\alpha(u) \, \di u\\
				&\lesi t^{-\f{Q}{\alpha}}\simeq \f{1}{t^{Q/\alpha}}\left(1+\f{| \eta|_{\HH^N}}{t^{1/\alpha}}\right)^{-(Q+\alpha)},
	\end{aligned}
	\]
	where in the second line we used the change of variable $u=st^{-2/\alpha}$ and in the last inequality we used (v).
	We have proved that 
	\[
	G_\alpha(\eta,t)\lesi \f{1}{t^{Q/\alpha}}\left(1+\f{| \eta|_{\HH^N}}{t^{1/\alpha}}\right)^{-(Q+\alpha)}
	\] 
	for $\zeta\in \HH^N$ and  $t>0$, which is the upper bound in \eqref{A1}.
	
	It remains to prove the lower bound in \eqref{A1} for $G_\alpha(\eta,t)$. Indeed, by \eqref{eq- Gaussian of ptxy}, \eqref{eq-sub formula},  and (iv), 
	\[
	\begin{aligned}
		G_\alpha(\eta,t)&\gtrsim \int_{t^{2/\alpha}}^{\vc}\f{t}{s^{(Q+\alpha)/2}}\exp\left(-\f{| \eta|_{\HH^N}^2}{c_1s}\right)  \f{ \di s}{s}.
	\end{aligned}
	\]
By \eqref{eq-sub formula} and (iv),
\[
\begin{aligned}
	G_\alpha(\eta,t)&\gtrsim \int_{t^{2/\alpha}}^{\vc}\f{t}{s^{(Q+\alpha)/2}}\exp\left(-\f{| \eta|_{\HH^N}^2}{c_1s}\right)  \f{ \di s}{s}\\
	&\gtrsim \int_{t^{2/\alpha}+|\eta|_{\HH^N}^2}^{\vc} \f{1}{s^{Q/2}}\f{t}{s^{1+\alpha/2}}\, \di s\\
	&\simeq \min\left\{ t^{-\f{Q}{\alpha}},\f{t}{|\eta|_{\HH^N}^{Q+\alpha}}\right\}=\min\left\{ t^{-\f{Q}{\alpha}},t^{-\f{Q}{\alpha}}\left(\f{t^{\f{1}{\alpha}}}{|\eta|_{\HH^N}}\right)^{Q+\alpha}\right\}\\
	&\simeq \f{1}{t^{Q/\alpha}}\left(1+\f{| \eta|_{\HH^N}}{t^{1/\alpha}}\right)^{-(Q+\alpha)}. 
\end{aligned}
\]
	This completes our proof.
\end{proof}

\begin{lemma}\label{lem-time derivaitive of heat kernel}
	For each $t>0$ and for any $\epsilon\in (0,\alpha)$ there exist $C, C'>0$ such that 
	\[
		|\partial_t G_\alpha(\eta,t)|\le 
		 \left\{
	\begin{aligned}
		&\f{C}{t^{Q/\alpha+1}}\left(1+\f{|\eta|_{\HH^N}}{t^{{1/\alpha}}}\right)^{-(Q+\alpha-\epsilon)} \quad
		&& \mbox{if} \quad \alpha \in (0,2),\\
		&\f{C}{t^{Q/2+1}}\exp\left( -\f{|\eta|_{\HH^N}^2}{C't}\right) \quad
		&& \mbox{if} \quad  \alpha=2,
	\end{aligned}
	\right.
	\]
	for $\eta\in \HH^N$ and $t>0$.
	Consequently, for $t>0$ we have
	\[
	|t\Lambda_\alpha [e^{t\Lambda_\alpha}f](\eta)|\lesi \mathcal Mf(\eta)
	\]
	for all $\eta\in \HH^N$ and $f\in L^1_{\rm loc}(\HH^N)$.
\end{lemma}
\begin{proof}
	The upper bounds for $|\partial_tG_\alpha(\eta,t)|$ are just a direct consequence of \cite[Lemma~2.5]{CD} and the upper bound \eqref{A1}. This together with Lemma~\ref{lem-maximal function} yields
	\[
	|t\Lambda_\alpha[e^{t\Lambda_\alpha}f](\eta)|
 =|t\partial_t[e^{t\Lambda_\alpha}f](\eta)|
 \lesi \mathcal Mf(\eta)
	\]
	for all $\eta\in \HH^N$,  $t>0$, and $f\in L^1_{\rm loc}(\HH^N)$, as desired.
	This completes our proof.
\end{proof}

\begin{lemma}
	\label{lem-maximal function domination}
	Let $N\ge1$.
	For every $t>0$, we have 
	\[
	\|e^{t\Lambda_\alpha}f\|_{L^\infty(\HH^N)}\le \|f\|_{L^\infty(\HH^N)}
	\]
	for all $f\in L^\vc(\HH^N)$.
\end{lemma}
\begin{proof}
	For $t>0$, $f\in L^\vc(\HH^N)$, and  $\eta\in \HH^N$ we have
	\[
	\begin{aligned}
		|e^{t\Lambda_\alpha}f(\eta)|&\le \int_{\HH^N}G_\alpha(\zeta^{-1}\circ \eta,t)|f(\zeta)| \, \di \zeta\\
		&\le \|f\|_{L^\infty(\HH^N)} \int_{\HH^N}G_\alpha(\zeta^{-1}\circ \eta,t) \, \di\zeta =  \|f\|_{L^\infty(\HH^N)},
	\end{aligned}
	\]
	where in the last inequality we used \eqref{eq:G5}.
	This completes our proof.
\end{proof}

\begin{lemma}
	\label{lem- domain}
	Let  $N\ge1$ and $\alpha\in (0,2]$. Then  $\Lambda_\alpha f\in L^\vc(\HH^N)$ for every $f\in C^\vc_c(\HH^N)$.
	
\end{lemma}
\begin{proof}
	The case $\alpha =2$ is straightforward. We only provide the proof for $\alpha\in (0,2)$.
	Let 
	 $f\in C^\vc_c(\HH^N)$. Then we have
	\[
	\Lambda_\alpha f  =   \frac{1}{\Gamma(1- \alpha/2)} \int_{0}^{\vc} t^{-\frac{\alpha}{2}} \Lambda_2 e^{t\Lambda_2} f \, \di t.
	\]
	This together with \eqref{eq:IBP}, Lemmas~\ref{lem-time derivaitive of heat kernel} and \ref{lem-maximal function domination}, and the boundedness of the maximal function $\mathcal M$ implies
	\[
	\begin{aligned}
	\|\Lambda_\alpha f\|_{L^\infty(\HH^N)}
	&\le \int_0^1 t^{-\frac{\alpha}{2}} \|e^{t\Lambda_2} \Lambda_2 f\|_{L^\infty(\HH^N)} \, \di t+ \int_1^\infty t^{-\frac{\alpha}{2}} \|t\Lambda_2 e^{t\Lambda_2 }f\|_{L^\infty(\HH^N)} \, \f{\di t}{t}\\
	&\lesi  \|\Lambda_2 f\|_{L^\infty(\HH^N)} +\|\mathcal Mf\|_{L^\infty(\HH^N)}\\
	&\lesi \|\Lambda_2 f\|_{L^\infty(\HH^N)} +\|f\|_{L^\infty(\HH^N)} <\infty.
\end{aligned}
	\]
	This completes our proof.
\end{proof}


\begin{lemma} 
	\label{lem - limit etL}
	 Let $N\ge1$ and $\alpha\in (0,2]$. Then we have
	\begin{enumerate}[\rm (a)]
		\item We have
		\[
		\lim_{t\to 0^+}\|e^{t\Lambda_\alpha}f-f\|_{L^\infty(\HH^N)} = 0 
		\]
		for $f\in C_c(\HH^N)$.
		
		\item Let $t>0$. Then
		\[
		\lim_{\tau\to 0^+}\|e^{(t-\tau)\Lambda_\alpha}f-e^{t\Lambda_\alpha}f\|_{L^\infty(\HH^N)} = 0
		\]
		for $f\in C_c(\HH^N)$.
	\end{enumerate}
\end{lemma}

\begin{proof}
	Since the proof of (a) is similar to (b) (even easier), we only give the proof of (b).
	 
We prove (b) for $f\in C^\vc_c(\HH^N)$. Indeed, for $0<\tau<t$ we have 
\[
e^{(t-\tau)\Lambda_\alpha}f-e^{t\Lambda_\alpha}f =-\int_{t-\tau}^t\Lambda_\alpha e^{s\Lambda_\alpha}f \, \di s = \int_{t}^{t-\tau}e^{s\Lambda_\alpha}(-\Lambda_\alpha) f \, \di s,
\]
which implies
\[
\|e^{(t-\tau)\Lambda_\alpha}f-e^{t\Lambda_\alpha}f \|_{L^\infty(\HH^N)} \le \int_{t-\tau}^t\|e^{s\Lambda_\alpha}\Lambda_\alpha f\|_{L^\infty(\HH^N)} \, \di s.
\]
By Lemmas~\ref{lem-maximal function domination} and \ref{lem- domain}, we further imply
\[
\|e^{(t-\tau)\Lambda_\alpha}f-e^{t\Lambda_\alpha}f \|_{L^\infty(\HH^N)} \le \tau\|\Lambda_\alpha f\|_{L^\infty(\HH^N)} < \infty.
\]
Therefore,
\[
\lim_{\tau\to 0^+}\|e^{(t-\tau)\Lambda_\alpha}f-e^{t\Lambda_\alpha}f\|_{L^\infty(\HH^N)}=0 
\]
for $f\in C^\vc_c(\HH^N)$.
Assume that  $f\in C_c(\HH^N)$. Then for any $\epsilon>0$ we can find $g\in C^\vc_c(\HH^N)$ such that 
\[
\|f-g\|_{L^\infty(\HH^N)}< \epsilon.
\]
This together with Lemma~\ref{lem-maximal function domination} implies
\begin{equation*}
\begin{split}
&\|e^{(t-\tau)\Lambda_\alpha}f-e^{t\Lambda_\alpha}f\|_{L^\infty(\HH^N)}\\
&\le \|e^{(t-\tau)\Lambda_\alpha}(f - g)\|_{L^\infty(\HH^N)}+\|e^{(t-\tau)\Lambda_\alpha}g- e^{t\Lambda_\alpha}g\|_{L^\infty(\HH^N)}+\|e^{t\Lambda_\alpha}g- e^{t\Lambda_\alpha}f\|_{L^\infty(\HH^N)}    \\
	&\le  \|f-g\|_{L^\infty(\HH^N)}  +\|e^{(t-\tau)\Lambda_\alpha}g- e^{t\Lambda_\alpha}g\|_{L^\infty(\HH^N)}   +\|g-f\|_{L^\infty(\HH^N)}\\
	&\le 2\|f-g\|_{L^\infty(\HH^N)} +\|e^{(t-\tau)\Lambda_\alpha}g- e^{t\Lambda_\alpha}g\|_{L^\infty(\HH^N)},
\end{split}
\end{equation*}
which implies
\[
\begin{aligned}
	\lim_{\tau\to 0^+}\|e^{(t-\tau)\Lambda_\alpha}f-e^{t\Lambda_\alpha}f\|_{L^\infty(\HH^N)} &\le 2\epsilon +\lim_{\tau\to 0^+}\|e^{(t-\tau)\Lambda_\alpha}g- e^{t\Lambda_\alpha}g\|_{L^\infty(\HH^N)}\le 2\epsilon
\end{aligned}
\]
for all $\epsilon>0$.
Therefore,
\[
\lim_{t\to 0^+}\|e^{(t-\tau)\Lambda_\alpha}f- e^{t\Lambda_\alpha}f\|_{L^\infty(\HH^N)}=0
\]
for all $f\in C_c(\HH^N)$.
This completes our proof.
\end{proof}

For any Radon measure $\mu$ on $\HH^N$, we define
\[
e^{t\Lambda_\alpha}\mu(\eta) = \int_{\HH^N} G_{\alpha}(\zeta^{-1}\circ \eta, t)d\mu(\zeta)
\]
for $\eta\in \HH^N$ and $t>0$.

We have the following estimate.
\begin{lemma}\label{lem 1}
Let $N\ge 1$ and $\alpha\in (0,2]$. For any Radon measure $\mu$ on $\HH^N$, there exists a constant $C>0$ such that
	\[
	\|e^{t\Lambda_\alpha}\mu\|_{L^\infty(\HH^N)} \le C  t^{-\frac{Q}{\alpha}}\sup_{\eta\in \HH^N}\mu(B(\eta,t^\frac{1}{\alpha}))
	\]
	for $t>0$.
\end{lemma}
\begin{proof}
	For $\eta\in \HH^N$ and $t>0$, by applying \eqref{A1} we have
	\[
	\begin{aligned}
		|e^{t\Lambda_\alpha}\mu(\eta)|&\le \int_{\HH^N} \f{1}{t^{Q/\alpha}} g_\alpha\left(\f{|\zeta^{-1}\circ \eta|_{\HH^N}}{t^{1/\alpha}}\right) \, \di\mu(\zeta)\\
		&\le \int_{B(\eta,2 t^\frac{1}{\alpha})} \f{1}{t^{Q/\alpha}} \, \di\mu(\zeta)\\
		&+\sum_{j\ge 1} \int_{\{2^{j}t^\frac{1}{\alpha}<\dist_\mathbb{H}(\eta,\zeta)\le 2^{j+1}t^\f{1}{\alpha}\}} \f{1}{t^{Q/\alpha}} g_\alpha\left(\f{\dist_\mathbb{H}(\eta,\zeta)}{t^{1/\alpha}}\right) \,\di\mu(\zeta)\\
		&\lesi \int_{B(\eta,2 t^\frac{1}{\alpha})}\f{1}{t^{Q/\alpha}} \, \di\mu(\zeta)+\sum_{j\ge 1} \int_{B(\eta,2^{j+1}t^\frac{1}{\alpha})} \f{2^{-j(Q+\alpha)}}{t^{Q/\alpha}} \,   \di\mu(\zeta)\\
		&\lesi \sum_{j\ge 1} t^{-\frac{Q}{\alpha}}\int_{B(\eta,2^{j}t^\f{1}{\alpha})}  2^{-j(Q+\alpha)}  \, \di\mu(\zeta).
	\end{aligned}
	\]
	Applying Lemma~\ref{lem-covering lemma}, for each $j\ge 1$ we can cover the ball $B(\eta,2^{j}t^{1/\alpha})$ by at most $C2^{jQ}$ balls whose radii all equal to $t^{1/\alpha}$. Consequently, for each $j\ge 1$,
	\[
	\int_{B(\eta,2^{j}t^\frac{1}{\alpha})} \,   \di\mu(\zeta)\lesi  2^{jQ}\sup_{\xi \in \HH^N} \mu(B(\xi,t^\frac{1}{\alpha})).
	\]  
	Therefore, for $\eta\in \HH^N$ and $t>0$, 
	\[
	\begin{aligned}
		|e^{t\Lambda_\alpha}\mu(\eta)| &\lesi \sum_{j\ge 1} 2^{-j\alpha}t^{-\frac{Q}{\alpha}} \sup_{\xi \in \HH^N} \mu(B(\xi,t^\frac{1}{\alpha}))\lesi  t^{-\frac{Q}{\alpha}} \sup_{\xi \in \HH^N} \mu(B(\xi,t^\frac{1}{\alpha})).
	\end{aligned}
	\]
	This completes our proof.
\end{proof}

\subsection{Preliminary lemmas.} 
At the end of Section 2, we provide some lemmas to prove the solvability.

\begin{lemma}\label{lem - supersolutions imply solution}
	Let $\mu$ be a nonnegative Radon measure on $\HH^N$ and $T\in(0, \infty]$. Assume that there exists a supersolution $v$ of  problem \eqref{eq:Fujita} with \eqref{eq:Fujitaini} in $\HH^N\times[0,T)$. Then there exists a minimal solution of problem \eqref{eq:Fujita} with \eqref{eq:Fujitaini} in $\HH^N\times[0,T)$.
\end{lemma} 

\begin{proof} The proof is quite standard. See e.g. \cites{HI18, HIT23, IKS16, RS13}. However, we would like to provide it for the sake of completeness.
Define $\{u_n\}_{n\ge 1}$ as follows. Set $u_1(\eta, t) := [e^{t\Lambda_\alpha}\mu](\eta)$ and define
\begin{equation}\label{eq- sequence un}
	u_n(\eta, t) := [e^{t\Lambda_\alpha}\mu](\eta) + \int_0^t [e^{(t-s)\Lambda_\alpha}u_{n-1}(s)^p](\eta) \, \di s 
\end{equation}
for $n\ge 2$.
Let $v$ be a supersolution of problem \eqref{eq:Fujita} with \eqref{eq:Fujitaini} in $\HH^N\times[0,T)$, where $T\in(0, \infty]$. Then it follows inductively that
\[
0 \leq u_1(\eta,t) \leq u_2(\eta,t) \leq \cdots \leq u_n(\eta,t) \leq \cdots \leq v(\eta,t) < \infty
\]
for a.a.~$\eta \in \HH^N$ and $t \in (0, T)$. It follows that
\[
u(\eta,t) := \lim_{n \to \infty} u_n(\eta,t) \leq v(\eta,t)
\]
for a.a.~$\eta \in \HH^N$ and $t \in (0, T)$. This, along with \eqref{eq- sequence un}, implies that $u$ is a solution of problem \eqref{eq:Fujita} with \eqref{eq:Fujitaini}. Since any solution is also a supersolution, if $\tilde u$ is another solution of problem \eqref{eq:Fujita} with \eqref{eq:Fujitaini} then a similar argument also shows that $u\le \tilde u$. Hence, $u$ is a minimal solution. 
This completes our proof.
\end{proof}

The key to the proof of Theorem~\ref{main thm} in the case of $\alpha \in (0,2)$ is the following lemma on the existence of solutions of ordinary differential equations.
This idea comes from \cite{LS21}.

\begin{lemma}
\label{Lemma:XODE}
Let $f$ be a nonnegative measurable function on $(0,T)$ for some $T>0$.
Assume that 
\begin{equation}
\label{eq:XODE}
\infty> f(t) \ge a_1 + a_2 \int_{t_*}^t s^{-a} f(s)^b \, \di s \quad \mbox{for a.a.} \quad t\in (t_*,T),
\end{equation}
where $a_1,a_2>0$, $a\ge0$, $b>1$, and $t_*\in (0,T/2)$.
Then there exists $C=C(a,b)>0$ such that
\begin{equation*}
a_1\le Ca_2^{-\frac{1}{b-1}} t_*^\frac{a-1}{b-1}.
\end{equation*}
In addition, if $a=1$, then 
\begin{equation*}
a_1\le (a_2(b-1))^{-\frac{1}{b-1}} \left[\log\frac{T}{2t_*}\right]^{-\frac{1}{b-1}}.
\end{equation*}
\end{lemma}
We remark that assumption \eqref{eq:XODE} implies that there exists a unique solution of the ordinary differential equation
$f'(t) = a_2 t^{-a}f(t)^b$ on $ (t_*,T)$ with $f(t_*) = a_1$.
The proof of Lemma~\ref{Lemma:XODE} is done by analyzing the solution of this equation.
For details of the proof, see e.g. \cite[Lemma 2.5]{HIT23}.

\section{Initial trace.}
In this section we show the existence and uniqueness of the initial trace and prove \eqref{eq- mainthm 1}.
The proof  follows the arguments in \cite{HI18}.

\begin{lemma}\label{lem 2} Let $u$ be a solution of \eqref{eq:Fujita} in $\HH^N\times(0,T)$, where $T\in(0, \infty)$. Then
\begin{equation}\label{eq1-lem2}
\operatorname*{ess\sup}_{0 < t < T-\epsilon} \int_{B(0,R)} u(\zeta,t) \, \di\zeta  < \infty
\end{equation}
for all $R>0$ and $0<\epsilon<T$.
Furthermore, there exists a unique Radon measure $\nu$ on $\HH^N$ such that
\begin{equation}\label{eq2-lem2}
\operatorname*{ess\,lim}_{t \to 0^+} \int_{\HH^N} u(\zeta,t) \phi(\zeta) \, \di\zeta  = \int_{\HH^N} \phi(\zeta) \, \di\nu(\zeta)
\end{equation}
for all $\phi \in C_c(\HH^N)$.
\end{lemma}
\begin{proof} 

Let $R, \epsilon >0$.  Then, from \eqref{eq - additional condition solution} and  the nonnegativity of $G_\alpha$, there exists $\eta_0 \in B(0,R)$ such that
\[
\begin{aligned}
\infty > u(\eta_0,t) \ge  
\int_{B(\eta_0, 2R)}  G_\alpha( \zeta^{-1}\circ\eta_0,t-\tau) u(\zeta,\tau) \, \di\zeta 	
\end{aligned}
\]
	for a.a.~$\tau \in (0,T-\epsilon)$ and $t \in (T-\epsilon/2,T)$.
	It follows that 
	\[
	\infty >    
	\inf_{\xi\in B(\eta_0,2R), \frac{\epsilon}{2}<s<T}G_\alpha(\xi^{-1}\circ \eta_0,s) \int_{B(\eta_0, 2R)}  u(\zeta,\tau) \, \di\zeta .
	\]
	From \eqref{A1}, we have
	\[
	\inf_{\xi\in B(\eta_0,2R), \epsilon<s<T}G_\alpha(\xi^{-1}\circ \eta_0,s)\ge \f{C_1}{T^{Q/\alpha}}g_\alpha\left(\f{CR}{c_1\epsilon^{1/\alpha}}\right),
	\] 
	which implies
	\[
	\int_{B(\eta_0, 2R)}  u( \zeta,\tau) \, \di\zeta <C(T, R,\epsilon)
	\]
	for a.a.~$\tau \in (0, T-\epsilon)$. 
	It follows  \eqref{eq1-lem2} since $B(0,R)\subset B(\eta_0,  2R)$.

	\bigskip
	
	We now take care of \eqref{eq2-lem2}. It suffices to prove \eqref{eq2-lem2} for $0\le \phi \in C_c(\HH^N)$. From \eqref{eq1-lem2}, the Riesz representation theorem \cite[Theorem 7.2.8]{Cohn}  and the weak compactness of Radon measures \cite[4.4 THEOREM]{Simon}, we can find a sequence $\{t_j\}$ with $\lim_{j\to\infty} t_j = 0$ and a  nonnegative Radon measure $\nu$ on $\HH^N$ such that
	\begin{equation}\label{limit for nu}
	\lim_{j\to\infty} \int_{\HH^N} u(\zeta,t_j) \phi(\zeta) \, \di\zeta  = \int_{\HH^N} \phi(\zeta) \, \di\nu(\zeta) 
	\end{equation}
	for all $\phi \in C_c(\HH^N)$. 
	
	For the uniqueness of the Radon measure $\nu$, we assume that there exist a sequence $\{s_j\}$ with $\lim_{j\to\infty} s_j = 0$ and a nonnegative Radon measure $\nu'$ on $\HH^N$ such that
	\begin{equation}\label{limit for nu'}
	\lim_{j\to\infty} \int_{\HH^N} u(\zeta,s_j) \phi(\zeta) \, \di\zeta  = \int_{\HH^N} \phi(\zeta) \, \di\nu'(\zeta)
	\end{equation}
	for all $\phi \in C_c(\HH^N)$. Taking a subsequence if necessary,  we might assume that  $t_j > s_{j}$ for $j = 1, 2, \dots$. From \eqref{eq - additional condition solution},
	\[
	u(\eta,t_j) \geq \int_{\HH^N} G_\alpha(  \zeta^{-1}\circ \eta,t_j-s_j) u( \zeta,s_j) \, \di\zeta , \quad j = 1, 2, \dots
	\]
	for a.a.~$\eta\in\mathbb{H}^N$.
	From the above inequality,  the fact $(\zeta^{-1}\circ \eta)^{-1} = \eta^{-1}\circ \zeta$, and \eqref{eq:G2}, we have
	\[
	\begin{aligned}
	&\int_{\HH^N} u(\eta, t_j) \phi(\eta) \, \di\eta\\
	 &\geq \int_{\HH^N} \left( \int_{\HH^N} G_\alpha( \zeta^{-1}\circ \eta,t_j-s_j) \phi(\eta) \, \di\eta \right) u(\zeta,s_j) \, \di\zeta \\
	&=\int_{\HH^N} \left( \int_{\HH^N} G_\alpha(\eta^{-1}\circ \zeta,t_j-s_j) \phi(\eta) \, \di\eta \right) u(\zeta,s_j) \, \di\zeta \\
	&= \int_{\HH^N} [e^{(t_j - s_{j})\Lambda_\alpha}\phi](\zeta) u(\zeta,s_j) \, \di\zeta  \\
	&\geq \int_{B(0,R)} [e^{(t_j - s_{j})\Lambda_\alpha}\phi](\zeta) u(\zeta,s_j) \, \di\zeta  \\
	&\geq \int_{B(0,R)} \phi(\zeta) u(\zeta,s_j) \, \di\zeta  - \|e^{(t_j - s_{j})\Lambda_\alpha}\phi - \phi\|_{L^\infty(B(0,R))} \int_{B(0,R)} u(\zeta,s_j) \, \di\zeta .
\end{aligned}
\]
Letting $j \to \infty$, by Lemma~\ref{lem - limit etL},  \eqref{eq1-lem2}, \eqref{limit for nu}, and \eqref{limit for nu'}  we obtain
	\[
	\int_{\HH^N} \phi \, \di\nu \geq \int_{\HH^N} \phi \, \di\nu'
	\]
	for all $\phi \in C_c(\HH^N)$.	
	Similarly, it follows that
	\[
	\int_{\HH^N} \phi \, \di\nu' \geq \int_{\HH^N} \phi \, \di\nu
	\]
	for all $\phi \in C_c(\HH^N)$.	
	This completes our proof.
\end{proof}

\begin{lemma}
	\label{lem 3}
	Let $\mu$ be nonnegative Radon measure on $\HH^N$. Let $u$ be a solution of problem \eqref{eq:Fujita} with \eqref{eq:Fujitaini} in $\HH^N\times[0,T)$, where $T\in(0,\infty)$. Then
	\begin{equation}
		\label{eq 1 - limit identity}
	\operatorname*{ess\,lim}_{t \to 0^+} \int_{\HH^N} u(\eta,t) \phi(\eta) \, \di\eta =
	\int_{\HH^N} \phi(\eta) \, \di\mu(\eta) 
	\end{equation}
	for all $\phi \in C_c(\HH^N)$.
\end{lemma}
\begin{proof}
	
	Without loss of generality we need only to verify \eqref{eq 1 - limit identity} with $0\le \phi \in C_c(\HH^N)$. Let $R \ge 1$  such that $\operatorname{supp} \phi \subset B(0,R)$. 
	By Lemma~\ref{lem 2}, we can find a unique Radon measure $\nu$ on $\HH^N$ such that
	\begin{equation}\label{eq - nu}
	\operatorname*{ess\,lim}_{t \to 0^+} \int_{\HH^N} u(\eta, t) \phi(\eta) \, \di\eta  = \int_{\HH^N} \phi(\eta) \, \di\nu(\eta)
	\end{equation}
	for all $0\le \phi \in C_c(\HH^N)$.
	Furthermore, from \eqref{eq - additional condition solution} we have
	\begin{equation}\label{eq1-lemma 3}
	\begin{aligned}
		\int_{\HH^N} u(\eta,t) \phi(\eta) \, \di\eta
		&= \int_{\HH^N} \int_{\HH^N} G_\alpha(\zeta^{-1}\circ \eta,t-\tau) u(\zeta,\tau) \phi(\eta) \, \di\eta \di\zeta  \\
		&\quad + \int^{t}_{\tau} \int_{\HH^N} \int_{\HH^N} G_\alpha( \zeta^{-1}\circ \eta,t-s) u( \zeta,s)^p\phi(\eta) \, \di\eta \di\zeta  \di s \\
		&\geq \int_{B(0,R)} \int_{\HH^N} G_\alpha(\zeta^{-1}\circ \eta,t-\tau) u( \zeta,\tau) \phi(\eta) \, \di\eta \di\zeta  \\
		&\quad + \int^{t}_{\tau} \int_{\HH^N} \int_{\HH^N} G_\alpha(\zeta^{-1}\circ \eta,t-s) u(\zeta,s)^p\phi(\eta) \, \di\eta \di\zeta  \di s 
	\end{aligned}
	\end{equation}
	for a.a.~$0 < \tau < t < T$ and $\tau \in (0,T/2)$. 
	On the other hand, since $(\zeta^{-1}\circ \eta)^{-1}=\eta^{-1}\circ \zeta$,  by \eqref{eq:G2} we have
	\begin{equation}\label{eq2-lemma 3}
	\begin{aligned}
		&\int_{B(0,R)} \int_{\HH^N} G_\alpha(\zeta^{-1}\circ \eta,t-\tau) u(\zeta,\tau) \phi(\eta) \, \di\eta \di\zeta  \\
		&= \int_{B(0,R)} u(\zeta,\tau)\int_{\HH^N} G_\alpha(\eta^{-1}\circ \zeta,t-\tau) \phi(\eta) \, \di\eta   \di\zeta  \\
		&= \int_{B(0,R)} [e^{(t-\tau)\Lambda_\alpha}\phi](\zeta) u(\tau, \zeta) \, \di\zeta  \\
		&\geq \int_{B(0,R)} \phi(\zeta) u(\tau, \zeta) \, \di\zeta - \|e^{(t-\tau)\Lambda_\alpha}\phi - \phi\|_{L^\infty(B(0,R))} \int_{B(0,R)} u(\zeta,\tau) \, \di\zeta \\
		&= \int_{B(0,R)} \phi(\zeta) u( \zeta,\tau) \, \di\zeta  -\|e^{(t-\tau)\Lambda_\alpha}\phi - \phi\|_{L^\infty(B(0,R))}  \operatorname*{ess\sup}_{\tau\in(0,T/2)}\int_{B(0,R)} u(\zeta,\tau) \, \di\zeta 
	\end{aligned}
	\end{equation}
	for a.a.~$0 < \tau < t < T$ and $\tau \in (0,T/2)$.
This, in combination with \eqref{eq - nu} and Lemma~\ref{lem - limit etL}, yields
	\[
	\begin{aligned}
	&\operatorname*{ess\,liminf}_{\tau \to 0^+} \int_{B(0,R)} \int_{\HH^N} G_\alpha(\zeta^{-1}\circ \eta, t-\tau) u(\zeta, \tau) \phi(\eta) \, \di\eta \di\zeta   		\\
	& \geq \int_{\HH^N} \phi(\zeta) \, \di\nu(\zeta) - C\|e^{t\Lambda_\alpha}\phi  - \phi\|_{L^\infty(B(0,R))}
	\end{aligned}
	\]
	for some $C>0$ and  a.a.~$0 < t < T$. Taking this, \eqref{eq1-lemma 3}  and \eqref{eq2-lemma 3} into account, we obtain
	\[
	\begin{aligned}
		\int_{\HH^N} u(\eta,t) \phi(\eta) \, \di\eta  
		&\geq \int_{\HH^N} \phi(\zeta) \, \di\nu(\zeta) - C\|e^{t\Lambda_\alpha}\phi - \phi\|_{L^\infty(B(0,R))} \\
		&\quad + \int_{0}^{t} \int_{\HH^N} \int_{\HH^N} G_\alpha(\zeta^{-1}\circ \eta,t-s) u(\zeta,s)^p\phi(\eta) \, \di\eta \di\zeta  \di s
	\end{aligned}
	\]
	for a.a.~$0 < t < T$. 
	By invoking \eqref{eq - nu} and Lemma~\ref{lem - limit etL},
		\[
	\operatorname*{ess\,lim}_{t \to 0^+} \int_{0}^{t} \int_{\HH^N} \int_{\HH^N} G_\alpha(\zeta^{-1}\circ \eta,t-s) u(\zeta,s)^p\phi(\eta) \, \di\eta \di\zeta  \di s = 0.
	\]
	
This, together with \eqref{eq- defn sols of parabolic eqs}, implies that 
	\begin{equation*}
	\begin{aligned}
	&\operatorname*{ess\,lim}_{t \to 0^+}\int_{\HH^N} u(\eta,t) \phi(\eta) \, \di\eta\\
	 &=\operatorname*{ess\,lim}_{t \to 0^+}\int_{\HH^N} \int_{\HH^N} G_\alpha(\zeta^{-1}\circ \eta,t)  \phi(\eta)\, \di\eta \di\mu(\zeta)\\
	&=
	\operatorname*{ess\,lim}_{t \to 0^+}\int_{\HH^N} \int_{\HH^N} G_\alpha(\eta^{-1}\circ \zeta,t)  \phi(\eta) \, \di\eta  \di\mu(\zeta)\\
	&=
	\operatorname*{ess\,lim}_{t \to 0^+}\int_{\HH^N} [e^{t\Lambda_\alpha}\phi](\zeta)  \, \di\mu(\zeta)\\
	&=
	\operatorname*{ess\,lim}_{t \to 0^+}\int_{\HH^N} \left([e^{t\Lambda_\alpha}\phi](\zeta)-\phi(\zeta)\right) \, \di\mu(\zeta) + \int_{\HH^N} \phi(\zeta)  \, \di\mu(\zeta).
\end{aligned}
\end{equation*}
Hence, it suffices to prove the following
\[
\lim_{t\to 0^+}\int_{\HH^N} ([e^{t\Lambda_\alpha}\phi](\zeta)-\phi(\zeta))    \, \di\mu(\zeta)  = 0. 
\]
We consider two cases: $\alpha \in (0,2)$ and $\alpha=2$.

\textbf{Case 1: $\alpha \in (0,2)$.} In this case, by using \eqref{A1},
\[
\begin{aligned}
	[e^{t\Lambda_\alpha}\phi] (\zeta)
	&\lesi  \int_{B(0,R)}\f{1}{t^{Q/\alpha}}\left(1+\f{\dist_\mathbb{H}(\zeta,\xi)}{t^{1/\alpha}}\right)^{-Q-\alpha}\phi(\xi)  \, \di\xi \\
	&\lesi R^Q t(t^\f{1}{\alpha}+|\zeta|_{\HH^N})^{-Q-\alpha} \|\phi\|_{L^\infty(\HH^N)}\\
	&\lesi R^Q T |\zeta|_{\HH^N}^{-Q-\alpha} \|\phi\|_{L^\infty(\HH^N)}\\
	&\lesi R^QT \|\phi\|_{L^\infty(\HH^N)}(1+|\zeta|_{\HH^N})^{-Q-\alpha}
\end{aligned}
\]
for $\zeta\in \HH^N\backslash B(0,2R)$ and $0<t\le T/2$.
This, together with the fact $\|e^{t\Lambda_\alpha}\phi\|_{L^\infty(\HH^N)}\le \|\phi\|_{L^\infty(\HH^N)}$ (see Lemma~\ref{lem-maximal function domination}), implies
\[
	[e^{t\Lambda_\alpha}\phi] (\zeta)\le C(R,T) \|\phi\|_{L^\infty(\HH^N)} (1+|\zeta|_{\HH^N})^{-Q-\alpha}
\]
for all $\zeta\in \HH^N$ and  $0<t\le T/2$.

On the other hand, for  $\zeta\in \HH^N$ and $s\in (T/4,T/2)$, by \eqref{A1},
\[
\begin{aligned}
\int_{B(0, T^\f{1}{\alpha})}G_\alpha( \zeta^{-1}\circ \eta,s) \, \di\eta 
&\gtrsim \int_{B(0, T^\f{1}{\alpha})}\f{1}{s^{Q/\alpha}}\left(1+\f{\dist_\mathbb{H}(\eta,\zeta)}{s^{1/\alpha}}\right)^{-Q-\alpha} \, \di\eta\\
&\simeq\int_{B(0, T^\f{1}{\alpha})}\f{1}{s^{Q/\alpha}} \left(1+\f{|\zeta|_{\HH^N}}{s^{1/\alpha}}\right)^{-Q-\alpha} \, \di\eta\\
 &\ge C(T)(1+|\zeta|_{\HH^N})^{-Q-\alpha}.
\end{aligned}
\]

Therefore, for $s\in (T/4,T/2)$, by Lemma~\ref{lem 2} we have
\[
\begin{aligned}
	\int_{\HH^N}  (1+|\zeta|_{\HH^N})^{-Q-\alpha}  \, \di\mu(\zeta) 
	&\lesi \int_{\HH^N} \int_{B(0, T^\f{1}{\alpha})}G_\alpha(\zeta^{-1}\circ \eta,s)   \,\di\eta \di\mu(\zeta) \\
	&\lesi \int_{B(0, T^\f{1}{\alpha})} [e^{s\Lambda_\alpha}\mu](\eta) \, \di\eta\\
	&\lesi \sup_{s\in (T/4,T/2)}\int_{B(0, T^\f{1}{\alpha})} u(\eta,s) \, \di\eta<+\vc.
\end{aligned}
\]
This, in combination with Lemma~\ref{lem - limit etL} and the Lebesgue Domination theorem, implies that
\[
\lim_{t\to 0^+}\int_{\HH^N} ([e^{t\Lambda_\alpha}\phi](\zeta)-\phi(\zeta))   \, \di\mu(\zeta)  = 0 
\]
as desired and hence this completes the proof for the case $\alpha\in (0,2)$.

\medskip

\textbf{Case 2: $\alpha =2$.} The proof of this case has the same spirit as that of the case $\alpha\in (0,2)$. 
Since $\tau\mapsto \tau^{Q/2}\exp(-\tau/(2c_2))$ is bounded on $(0,\infty)$,
it follows from \eqref{A1} that,
\[
\begin{aligned}
	&[e^{t\Lambda_\alpha}\phi] (\zeta)\\
	&\le C_2 \int_{B(0,R)}\f{1}{t^{Q/2}}\exp\left(-\f{\dist_\mathbb{H}(\zeta,\xi)^2}{c_2t}\right) \phi(\xi)  \, \di\xi \\
	&= C_2 \int_{B(0,R)}\left[\f{\dist_\mathbb{H}(\zeta,\xi)^Q}{t^{Q/2}}\exp\left(-\f{\dist_\mathbb{H}(\zeta,\xi)^2}{2c_2t}\right) \right]\frac{1}{\dist_\mathbb{H}(\zeta,\xi)^Q}\exp\left(-\f{\dist_\mathbb{H}(\zeta,\xi)^2}{2c_2t}\right) \phi(\xi)  \, \di\xi \\
	&\lesi  R^{-Q}\int_{B(0,R)} \exp\left(-\f{\dist_\mathbb{H}(\zeta,\xi)^2}{2c_2t}\right) \phi(\xi) \, \di\xi \\
	&\lesi  \exp\left(-\f{|\zeta|_{\HH^N}^2}{2c_2t}\right) \|\phi\|_{L^\infty(\HH^N)}
	\le  \exp\left(-\f{8|\zeta|_{\HH^N}^2}{c_1T}\right) \|\phi\|_{L^\infty(\HH^N)}
\end{aligned}
\]
for  $\zeta\in \HH^N\backslash B(0, 2R)$ and $0<t\le {c_1T/(16c_2)}$.
This together with the fact $\|e^{t\Lambda_\alpha}\eta\|_{L^\infty(\HH^N)}\le C\|\eta\|_{L^\infty(\HH^N)}$ (see Lemma~\ref{lem-maximal function domination}), implies
\[
[e^{t\Lambda_\alpha}\phi](\zeta)\le C(R,T) \|\phi\|_{L^\infty(\HH^N)} \exp\left(-\f{8|\zeta|_{\HH^N}^2}{c_1T}\right)
\]
for all   $\zeta\in \HH^N$ and $0<t\le {c_1T/(16c_2)}$.
On the other hand, for  $\zeta\in \HH^N$ and $s\in (T/4,T/2)$, by \eqref{A1},
\[
\begin{aligned}
	\int_{B(0, T^\f{1}{\alpha})}G_\alpha(\zeta^{-1}\circ \eta,s) \, \di\eta &\ge C\int_{B(0, T^\f{1}{\alpha})}\f{1}{s^{Q/\alpha}}\exp\left(-\f{\dist_\mathbb{H}(\eta,\zeta)^2}{c_1s}\right) \, \di\eta\\
	&\ge C\int_{B(0, T^\f{1}{\alpha})}\f{1}{s^{Q/\alpha}} \exp\Big(-\f{2|\zeta|_{\HH^N}^2}{c_1s}\Big)\di\eta\\
	&\ge C\exp\Big(-\f{2|\zeta|_{\HH^N}^2}{c_1s}\Big)\ge C\exp\Big(-\f{8|\zeta|_{\HH^N}^2}{c_1T}\Big).\\
\end{aligned}
\]
Therefore, for $s\in (T/4,T/2)$, by Lemma~\ref{lem 2} we have
\[
\begin{aligned}
	 \int_{\HH^N}  \exp\Big(-\f{8|\zeta|_{\HH^N}^2}{c_1T}\Big)  \,   \di\mu(\zeta)  	
	 &\lesi \int_{\HH^N} \int_{B(0, T^\f{1}{\alpha})}G_\alpha(\zeta^{-1}\circ \eta,s)   \,\di\eta  \di\mu(\zeta) \\
	&= \int_{B(0, T^\f{1}{\alpha})} [e^{s\Lambda_\alpha}\mu](\eta) \, \di\eta\\
	&\lesi \sup_{s\in (T/4,T/2)}\int_{B(0, T^\f{1}{\alpha})} u(\eta,s) \, \di\eta<+\vc.
\end{aligned}
\]
This, in combination with Lemma~\ref{lem - limit etL} and the Lebesgue Domination theorem, implies that
\[
\lim_{t\to 0}\int_{\HH^N} ([e^{t\Lambda_\alpha}\phi](\zeta)-\phi(\zeta))\, \di\mu(\zeta)  = 0. 
\]
This completes the proof for the case $\alpha=2$.
The proof is complete.
\end{proof}

\section{Necessary conditions for the solvability.}
In this section we prove assertions (i)--(iii) in Theorem~\ref{main thm} and complete the proof of Theorem~\ref{main thm}.
Furthermore, we also prove Theorem~\ref{mainthm 2}.
The proof of Theorem~\ref{main thm}  is quite long and will be divided into two cases: $\alpha\in (0,2)$ and $\alpha=2$.

\bigskip

\noindent \underline{\textbf{The case $\alpha\in(0,2)$}}


\begin{lemma}
	\label{lem 5}
 Let $N\ge1$, $\alpha\in (0,2)$, and  $ p >1$. Let $u$ be a solution of \eqref{eq:Fujita} in $\HH^N\times(0,T)$, where $T \in (0,\infty)$. Then there exist positive constants $C$ and $\delta\in(0,1)$ depending only on $N$, $\alpha$ and $p$ such that
\[
\sup_{\xi\in \HH^N} \int_{B(\xi,\rho)} u(\zeta,\tau) \, \di\zeta  \leq 
\left\{
\begin{array}{ll}
\displaystyle{C\left[\log \left( e+\frac{ T^{1/\alpha}}{\rho} \right)\right]^{-\f{Q}{\alpha}}} & \mbox{if} \quad p=p_{\alpha,Q},\vspace{3pt}\\
C\rho^{Q-\frac{\alpha}{p-1}}& \mbox{if} \quad p\neq p_{\alpha,Q},\vspace{3pt}\\
\end{array}
\right.
\]
for all $\rho > 0$ with $0 < \rho^\alpha < \delta T$ and for a.a.~$\tau \in (0, \rho^\alpha)$.
\end{lemma}
\begin{proof}
Let $\delta>0$ be  sufficiently small. Let
\begin{equation}\label{eq- rho T}
	0 < \rho < (\delta T)^\f{1}{\alpha}. 
\end{equation}
We set $v(\eta,t) := u(\eta, t + (2\rho)^\alpha)$ for a.a.~$\eta\in \HH^N$ and $t \in (0, T - (2\rho)^\alpha)$. Then it follows from \eqref{eq - additional condition solution} that
\begin{equation}
\label{eq1-lem 5}
\begin{split}
	\infty>v(t, \eta) 
	&= \int_{\HH^N} G_\alpha(\zeta^{-1}\circ \eta,t-\tau)v(\tau, \zeta) \, \di\zeta + \int_{\tau}^{t} \int_{\HH^N} G_\alpha(\zeta^{-1}\circ \eta,t-s) v(s, \zeta)^p \, \di\zeta \di s  
\end{split}
\end{equation}
for a.a.~$\eta\in \HH^N$ and $0 < \tau < t < T - (2\rho)^\alpha$. We shall show that 
\begin{equation}
	\label{eq1-lem 45}
	\int_{\HH^N} G_\alpha(\zeta^{-1}\circ \xi,\tau) v(\zeta,\tau) \, \di\zeta<\vc
\end{equation}
for a.a.~$\xi \in \HH^N$ and $\tau \in (0, [T - (2\rho)^\alpha]/3)$. 
From \eqref{A1}, 
\[
G_\alpha(\zeta^{-1}\circ \eta,\tau) \gtrsim   G_\alpha(\zeta^{-1}\circ \xi,\tau)
\]
for all $\dist_\mathbb{H}(\eta,\xi)< \tau^{1/\alpha}$. Then, using \eqref{eq1-lem 5} with $t=2\tau$, we have
\[
\vc> v(\eta,t)\ge \int_{\HH^N} G_\alpha(\zeta^{-1}\circ \eta,\tau) v(\zeta,\tau) \, \di\zeta\ge \int_{\HH^N} G_\alpha(\zeta^{-1}\circ \xi,\tau) v( \zeta,\tau) \, \di\zeta
\]
for a.a.~$\eta\in\mathbb{H}^N$ with $\dist_\mathbb{H}(\eta,\xi) \le \tau^{1/\alpha}$ and $\tau \in (0, [T - (2\rho)^\alpha]/3)$. It follows \eqref{eq1-lem 45}.
Furthermore, by \eqref{eq - additional condition solution}, \eqref{eq:G4}, and \eqref{eq1-lem 5},
\[
\begin{aligned}
	&v(\eta,t) - \int_{\tau}^{t} \int_{\HH^N} G_\alpha(\zeta^{-1}\circ \eta,t-s) v(\zeta,s)^p \, \di\zeta \di s\\
	 &= \int_{\HH^N} G_\alpha(\zeta^{-1}\circ \eta,t-\tau) u( \zeta,\tau + (2\rho)^\alpha) \, \di\zeta\\
	&\ge \int_{\HH^N} \int_{\HH^N}G_\alpha( \zeta^{-1}\circ \eta,t-\tau) G_\alpha( (\eta')^{-1}\circ \zeta,(2\rho)^\alpha)u(\eta',\tau) \, \di \eta' \di\zeta\\
	&= \int_{\HH^N}  G_\alpha( (\eta')^{-1}\circ \eta,t-\tau+ (2\rho)^\alpha)   u(\eta',\tau) \, \di \eta',
\end{aligned}
\]
which implies for $\xi \in \HH^N$,
\[
\begin{aligned}
	\int_{\HH^N}  G_\alpha(\eta^{-1}\circ \xi,t)v(\eta,t) \, \di\eta
	&\ge \int_{\HH^N}  G_\alpha((\eta')^{-1}\circ \xi,2t-\tau+ (2\rho)^\alpha) u(\eta',\tau) \, \di \eta'\\
	&  + \int_{\tau}^{t} \int_{\HH^N} G_\alpha(\zeta^{-1}\circ \xi,2t-s) v(\zeta,s)^p \, \di\zeta \di s\\
	&\ge  \inf_{\zeta\in B(\xi,\rho)}G_\alpha( \zeta^{-1}\circ \xi,2t-\tau+ (2\rho)^\alpha) \int_{B(\xi,\rho)}   u(\eta',\tau) \, \di \eta'\\
	&\int_{\rho^\alpha}^{t} \int_{\HH^N} G_\alpha( \zeta^{-1}\circ \xi,2t-s) v(\zeta,s)^p \, \di\zeta \di s\\
\end{aligned}
\]
for a.a.~$0 < \tau < \rho^{\alpha}<t < [T - (2\rho)^\alpha]/3$ and   $\xi\in \HH^N$. 
From \eqref{A1},
\[
\inf_{\zeta\in B(\xi,\rho)}G_\alpha(\zeta^{-1}\circ \xi,2t-\tau+ (2\rho)^\alpha)\ge c_0 (2t-\tau+ (2\rho)^\alpha)^{-\f{Q}{\alpha}}
\]
and
\[
\begin{aligned}
	G_\alpha(\zeta^{-1}\circ \xi,2t-s)
	\simeq \f{1}{t^{Q/\alpha}}g_\alpha\left(\f{\dist_\mathbb{H}(\xi,\zeta)}{(2t-s)^{1/\alpha}}\right) 
	&\gtrsim C_2  \left(\f{s}{2t}\right)^\f{Q}{\alpha} \f{1}{s^{Q/\alpha}}g_\alpha\left(\f{\dist_\mathbb{H}(\xi,\zeta)}{c_2s^{1/\alpha}}\right) \\
	&\ge  \left(\f{s}{2t}\right)^\f{Q}{\alpha}G_\alpha(\zeta^{-1}\circ \xi,s)
\end{aligned}
\]
for $0<s<t$.
This, in combination with \eqref{eq:G5} and Jensen's inequality, further implies
\[
\begin{aligned}
	\int_{\HH^N} G_\alpha(\eta^{-1}\circ \xi,t)v(\eta,t)  \, \di\eta 
	&\gtrsim (2t-\tau+ (2\rho)^\alpha)^{-\f{Q}{\alpha}} \int_{B(\xi,\rho)}   u(\eta',\tau) \, \di \eta'\\
	&+ \int_{\rho^\alpha}^{t} \left(\f{s}{2t}\right)^\f{Q}{\alpha} \int_{\HH^N} G_\alpha(\zeta^{-1}\circ \xi,s) v(\zeta, s)^p \, \di\zeta \di s \\
	&\gtrsim  t^{-\f{Q}{\alpha}} \int_{B(\xi,\rho)}   u(\eta',\tau) \, \di \eta'\\
	&+ \int_{\rho^\alpha}^{t} \left(\f{s}{2t}\right)^\f{Q}{\alpha} \left[\int_{\HH^N} G_\alpha(\zeta^{-1}\circ \xi,s) v(\zeta,s) \, \di\zeta\right]^p \, \di s \\
\end{aligned}
\]
for a.a.~$\xi\in \HH^N$ and $0 < \tau < \rho^{\alpha}<t < [T - (2\rho)^\alpha]/3$.

We now set
\[
m(t) := t^\frac{Q}{\alpha}\int_{\HH^N}  G_\alpha(\eta^{-1}\circ \xi,t) v(\eta,t) \, \di\eta, \quad M(\tau) := \int_{B(\xi, \rho)} u(\zeta,\tau) \, \di\zeta.
\]
Then we can  rewrite the above inequality as
\begin{equation*}
	\begin{aligned}
		\vc> m(t)
		&\ge  C  M(\tau) + C \int_{\rho^\alpha}^{t} s^{-\f{Q}{\alpha}(p-1) }m(s)^p \, \di s
	\end{aligned}
\end{equation*}
for a.a.~$0 < \tau < \rho^\alpha < t < [T - (2\rho)^\alpha]/3$. 

In the case of $p=p_{\alpha,Q}$,
note that $Q(p-1)/\alpha = 1$.
It follows from Lemma~\ref{Lemma:XODE} with $t_* = \rho^\alpha$ that
\begin{equation*}
\begin{split}
M(\tau) = \int_{B(\xi, \rho)} u(\zeta,\tau) \, \di\zeta \leq C \left[ \log \left( \frac{T}{2\rho^\alpha} \right) \right]^{-\f{Q}{\alpha}} \leq C \left[ \log \left( e+ \frac{ T}{\rho^\alpha} \right) \right]^{-\f{Q}{\alpha}}\end{split}
\end{equation*}
for all  $0 < \rho^\alpha < \delta T$
and a.a.~$0 < \tau < \rho^\alpha$. 
In the case of $p\neq p_{\alpha,Q}$, similarly we have
\begin{equation*}
\begin{split}
M(\tau) = \int_{B(\xi, \rho)} u(\zeta,\tau) \, \di\zeta \leq C \rho^{Q-\frac{\alpha}{p-1}}
\end{split}
\end{equation*}
for all  $0 < \rho^\alpha < \delta T$
and a.a.~$0 < \tau < \rho^\alpha$. 
This completes our proof.
\end{proof}

\begin{proof}
	[Proof of Theorem~{\rm \ref{main thm}} in the case of $\alpha\in (0,2)$]
	By Lemma~\ref{lem 2} we can find a unique Radon measure $\nu$
		satisfying \eqref{eq2-lem2}. Let $\xi \in \mathbb{H}^N$ and $\psi\in C_c(\mathbb{H}^N)$ be such that
		\begin{equation*}
		\psi=1 \, \,\mbox{in} \, \,B(\xi,\rho/2), \quad 0\le \psi\le1 \, \,\mbox{in} \, \,\mathbb{H}^N, \quad
		\psi = 0 \, \, \mbox{in}\,\, \mathbb{H}^N\setminus B(\xi,\rho).
		\end{equation*}
		 By Lemma~\ref{lem 2} there exists a unique nonnegative Radon measure $\nu$ on $\mathbb{H}^N$ such that
		\begin{equation*}
		\begin{split}
		\operatorname*{ess\,lim}_{\tau \to 0^+} \int_{B(\xi, \rho)} u(\zeta,\tau) \, \di\zeta
		 &\geq \operatorname*{ess\,lim}_{\tau \to 0^+} \int_{B(\xi, \rho)} u(\tau, \zeta) \psi(\zeta) \, \di\zeta \\
		&= \int_{\HH^N} \psi(\zeta)  \, \di\nu(\zeta) \geq \nu(B(\xi, \rho/2))
		\end{split}
		\end{equation*}
		for all $\xi\in \HH^N$ and $0 < \rho < (\theta T)^{1/\alpha}$, where $\theta>0$ is as in \eqref{eq- rho T}.
		This, together with Lemma~\ref{lem 5}, implies that
\[
\sup_{\xi\in \mathbb{H}^N}\nu(B(\xi, \rho/2)) \leq 
\left\{
\begin{array}{ll}
\displaystyle{C\left[\log \left( e+\frac{ T^{1/\alpha}}{\rho} \right)\right]^{-\f{Q}{\alpha}}} & \mbox{if} \quad p=p_{\alpha,Q},\vspace{3pt}\\
C\rho^{Q-\frac{\alpha}{p-1}}& \mbox{if} \quad p\neq p_{\alpha,Q},\vspace{3pt}\\
\end{array}
\right.
\]
for all $0 < \rho < (\theta T)^{1/\alpha}$.

Setting $\sigma := \theta^{-1/\alpha}\rho$, then by Lemma~\ref{lem-covering lemma} (a) we see that
there exist a positive integer $m$ depending only on $N$, $\alpha$, and $p>1$ and $\{\xi_j\}_{j=1}^m$ such that
\[
B(\xi, \sigma) \subset \bigcup_{j=1}^m B(\xi_j, \rho/2).
\]		
This implies that
\begin{equation*}
\begin{split}
\sup_{\xi\in \mathbb{H}^N}\nu(B(\xi, \sigma))  &\le m \sup_{\xi\in \mathbb{H}^N}\nu(B(\xi, \rho/2))\\
& \le \left\{
\begin{array}{ll}
\displaystyle{Cm\left[\log \left( e+\frac{ T^{1/\alpha}}{\rho} \right)\right]^{-\f{Q}{\alpha}}} & \mbox{if} \quad p=p_{\alpha,Q},\vspace{3pt}\\
Cm\rho^{Q-\frac{\alpha}{p-1}}& \mbox{if} \quad p\neq p_{\alpha,Q},\vspace{3pt}\\
\end{array}
\right.
\end{split}
\end{equation*}
By the definition of $\sigma$, we obtain
\[
\sup_{\xi\in \mathbb{H}^N}\nu(B(\xi, \sigma)) \leq 
\left\{
\begin{array}{ll}
\displaystyle{C\left[\log \left( e+\frac{ T^{1/\alpha}}{\sigma} \right)\right]^{-\f{Q}{\alpha}}} & \mbox{if} \quad p=p_{\alpha,Q},\vspace{3pt}\\
C\sigma^{Q-\frac{\alpha}{p-1}}& \mbox{if} \quad p\neq p_{\alpha,Q},\vspace{3pt}\\
\end{array}
\right.
\]
for all $0 < \sigma < T^{1/\alpha}$.
This is the desired inequality. Thus, Theorem~\ref{main thm} in the case of $\alpha\in(0,2)$ follows.
\end{proof}

\bigskip

\noindent \underline{\textbf{The case $\alpha=2$}}
In this case, we consider solutions of problem \eqref{eq:Fujita} with \eqref{eq:Fujitaini} in the following weak framework.

\begin{definition}
	\label{Definition:weak}
	Let $u$ be a nonnegative measurable function in $\mathbb{H}^N\times(0,T)$, where $T\in (0,\infty)$.
	We say that $u$ is a weak solution of problem  \eqref{eq:Fujita} in $\mathbb{H}^N\times[0,T)$ if $u \in L^p_{\rm loc}(\mathbb{H}^N\times[0,T))$ and $u$ satisfies 
	\begin{equation}
		\label{eq:weak}
		\begin{split}
			&\int_{0}^{T} \int_{\mathbb{H}^N} u(\eta,t)^p \varphi(\eta,t)\, \di \eta\di t + \int_{\mathbb{H}^N}  \varphi(\eta,0) \, \di\mu(\eta)\\
			&\qquad \qquad \qquad \qquad = \int_{0}^{T} \int_{\mathbb{H}^N} u(\eta,t) (-\partial_t-\Delta_\mathbb{H}) \varphi(\eta,t) \,\di\eta \di t
		\end{split}
	\end{equation}
	for all  $\varphi\in C^{2,1}_0(\mathbb{H}^N\times[0,T))$.
	If $u$ satisfies \eqref{eq:weak} with $=$ replaced by $\le$, then $u$ is said to be a  weak  supersolution.
\end{definition}
\begin{lemma}
	\label{Lemma:3.1}
	Let $u$ be a solution of problem \eqref{eq:Fujita} with \eqref{eq:Fujitaini} in $\mathbb{H}^N\times [0,T)$, where $T\in (0,\infty)$.
	Then $u$ is also a weak solution.
\end{lemma}

\begin{proof}
	Assume that problem \eqref{eq:Fujita} with \eqref{eq:Fujitaini} possesses a solution in $\mathbb{H}^N\times(0,T)$, where $T\in (0,\infty)$.
	
	We shall prove that $u\in L^p_{\rm loc} ( [0,T) \times \mathbb{H}^N)$.
	Let $\epsilon\in(0,T/2)$.  By \eqref{A1} and \eqref{eq - additional condition solution} we find $t\in(T-\epsilon,T)$ such that
	\begin{equation*}
		\begin{split}
			\infty&>u(\eta,t) \ge \int_{0}^{T-2\epsilon} \int_{\mathbb{H}^N} G(\zeta^{-1}\circ \eta,t-s) u( \zeta,s)^p\,\di\eta \di s\\
			&\ge C_1 \int_0^{T-2\epsilon} \int_{\mathbb{H}^N}  (t-s)^{-\frac{Q}{\alpha}} \exp\left(-\frac{  \dist_\mathbb{H}(\eta,\zeta)^2}{c_1(t-s)}\right) u(\zeta,s)^p\, \di\zeta \di s\\
			&\ge  C_1T^{-\frac{N}{2}} \int_0^{T-2\epsilon} \int_{\mathbb{H}^N}  \exp\left(-\frac{ c_1 \dist_\mathbb{H}(\eta,\zeta)^2}{c_1\epsilon}\right) u(\zeta,s)^p\,\di \zeta \di s
		\end{split}
	\end{equation*}
	for a.a.~$\eta\in{\mathbb{H}^N}$. Since $\epsilon\in(0,T/2)$ is arbitrary, we see that $u\in L_{\rm loc}^p(\mathbb{H}^N\times[0,T))$.

	Let $\varphi\in C^\infty_0(\mathbb{H}^N\times[0,T))$. 
	By the integral by parts and \eqref{eq:IBP} we have
	\begin{equation*}
		\begin{split}
			&\int_{\mathbb{H}^N} \varphi(\zeta,0) \, \di\mu(\zeta)\\
			&= \int_{\mathbb{H}^N} \left(\int_0^T\int_{\mathbb{H}^N}(\partial_t-\Delta_\mathbb{H})G( \zeta^{-1}\circ\eta,t)\cdot \varphi(\eta,t)\,\di\eta \di t+ \varphi(\zeta,0)\right) \,\di\mu(\zeta)\\
			&= \int_{\mathbb{H}^N} \int_0^T\int_{\mathbb{H}^N}G(\zeta^{-1}\circ\eta,t)(-\partial_t-\Delta_\mathbb{H})\varphi(\eta,t)\,\di\eta \di t \di\mu(\zeta)\\
			&= \int_0^T\int_{\mathbb{H}^N}\left(\int_{\mathbb{H}^N}G(\zeta^{-1}\circ\eta,t)\,\di\mu(\zeta)\right)(-\partial_t-\Delta_\mathbb{H})\varphi( \eta,t)\,\di \eta \di t\\
			&= \int_0^T\int_{\mathbb{H}^N}(-\partial_t-\Delta_\mathbb{H})\varphi(\eta,t)\,\di \eta \di t.
		\end{split}
	\end{equation*}
	Similarly, we have
	\begin{equation*}
		\begin{split}
			&\int_0^T\int_{\mathbb{H}^N} \varphi(\zeta,s)  u(\zeta,s)^p\, \di \zeta \di s\\
			&= \int_0^T\int_{\mathbb{H}^N} \left(\int_s^T\int_{\mathbb{H}^N}(\partial_t-\Delta_\mathbb{H})G(\zeta^{-1}\circ\eta,t-s)\cdot \varphi(\eta,t)\,\di \eta \di t + \varphi(\zeta,s)\right) u(\zeta,s)^p\,\di\zeta \di s\\
			&= \int_0^T\int_{\mathbb{H}^N} \left(\int_s^T\int_{\mathbb{H}^N}G(\zeta^{-1}\circ\eta,t-s)(-\partial_t-\Delta_\mathbb{H})\varphi(\eta,t)\,\di \eta \di t \right)u(\zeta,s)^p\,\di \zeta \di s\\
			&= \int_0^T\int_{\mathbb{H}^N} \left(\int_0^T\int_{\mathbb{H}^N}G(\zeta^{-1}\circ\eta,t-s)u( \zeta,s)^p\,\di \zeta \di s \right)(-\partial_t-\Delta_\mathbb{H})\varphi(\eta,t)\,\di \eta \di t.
		\end{split}
	\end{equation*}
	Then
	\begin{equation*}
		\begin{split}
			&\int_{0}^T\int_{\mathbb{H}^N} u(\eta,t)(-\partial_t-\Delta_\mathbb{H}) \varphi(\eta,t)\,\di \eta \di t\\
			&= \int_{0}^T\int_{\mathbb{H}^N} \left(\int_{\mathbb{H}^N}G(\zeta^{-1}\circ\eta,t)\, \di\mu(\zeta)+\int_0^T\int_{\mathbb{H}^N}G(\zeta^{-1}\circ\eta,t-s)u(\zeta,s)^p\,\di \zeta \di s\right)\\
			&\qquad\qquad\times(-\partial_t-\Delta_\mathbb{H}) \varphi( \eta,t)\,\di \eta \di t\\
			&= \int_{\mathbb{H}^N} \varphi(\zeta,0) \, \di\mu(\zeta)+ \int_0^T\int_{\mathbb{H}^N} \varphi(\zeta,s)  u(\zeta,s)^p\, \di\zeta \di s,
		\end{split}
	\end{equation*}
	which implies \eqref{eq:weak}. Then Lemma~\ref{Lemma:3.1} follows.
\end{proof}

\begin{proof}[Proof of Theorem~{\rm \ref{main thm}} in the case of $\alpha=2$]
	The proof follows the arguments in \cite[Theorem 1.2]{IKO20}.
	Let $u$ be a solution of  problem \eqref{eq:Fujita} with \eqref{eq:Fujitaini} in $\mathbb{H}^N\times[0,T)$, where $T\in(0,\infty)$.
	From the proof in the case of $\alpha\in(0,2)$,  it is sufficient to show that initial data $\mu$ satisfies assertions (i)--(iii) in Theorem~\ref{main thm}.
	It follows from Lemma~\ref{Lemma:3.1} that $u$ satisfies \eqref{eq:weak}.
	Let $\rho\in (0, \sqrt{T}/8)$.
	In what follows, denote
	\[
	\eta = (x,y,\tau)\in \mathbb{H}^N \quad\mbox{and}\quad \zeta = (x',y',\tau') \in \mathbb{H}^N.
	\]
	Note that
	\[
	\zeta^{-1} \circ \eta = (x-x', y-y', \tau-\tau' + 2(x\cdot y' - x'\cdot y)).
	\]
	
	Let 
	\[
	f(s) := e^{-\frac{1}{s}}\quad \mbox{if} \quad s>0, \qquad f(s) = 0\quad \mbox{if} \quad s\le0. 
	\]
	Set 
	\[
	F(s) := \frac{f(2-s)}{f(2-s)+ f(s-1)}.
	\]
	Then $F \in C^\infty([0,\infty))$ and 
	\begin{equation*}
		\begin{split}
			&F' (s) = \frac{-f'(2-s)f(s-1)- f(2-s)f'(s-1)}{[f(2-s)+f(s-1)]^2} \le 0 \quad \mbox{on} \quad [0,\infty),\\
			& F(s) = 1 \quad \mbox{on} \quad [0,1], \qquad F(s)=0 \quad \mbox{on} \quad [2,\infty).
		\end{split}
	\end{equation*}
	Set 
	\[
	F^*(s) := 0 \quad \mbox{on} \quad [0,1), \qquad F^*(s) := F(s) \quad \mbox{on} \quad [1,\infty).
	\]
	Since $p>1$, for any $k\in \mathbb{N}$, we can find $C_k>0$ such that 
	\begin{equation}
		\label{eq:3.1}
		|F^{(k)}(s)| \le C_k F^*(s)^\frac{1}{p} \quad \mbox{for all} \quad s\ge1.
	\end{equation}
	For any $R\in (0,T]$, we set
	\[
	\phi_R(\eta,t) := F\left(9\frac{|\eta|_{\mathbb{H}^N}^4 + t^2}{R^2}\right),\qquad 
	\phi^*_R(\eta,t) := F^*\left(9\frac{|\eta|_{\mathbb{H}^N}^4 + t^2}{R^2}\right),
	\]
	and for any $\zeta\in \mathbb{H}^N$, we set 
	\begin{equation*}
		\begin{split}
			&\psi_R(\eta,t) := \phi_R(\zeta^{-1}\circ\eta,t) = F\left(9\frac{\dist_\mathbb{H}(\eta,\zeta)^4 + t^2}{R^2}\right),\\
			&\psi^*_R(\eta,t) := \phi^*_R(\zeta^{-1}\circ\eta,t) = F^*\left(9\frac{\dist_\mathbb{H}(\eta,\zeta)^4 + t^2}{R^2}\right).\\
		\end{split}
	\end{equation*}
	For the simplicity of notation, set
	\[
	s_R(\eta,t) := 9\frac{|\eta|_{\mathbb{H}^N}^4 + t^2}{R^2} =  9\frac{\dist_\mathbb{H}(\eta,0)^4 + t^2}{R^2}.
	\]
	Note that
	\begin{equation}
		\label{eq:3.2}
		\Delta_\mathbb{H} \psi_R(\eta,t) = [\Delta_\mathbb{H}\phi_R](\zeta^{-1}\circ\eta,t)
	\end{equation}
	for all $\eta,\zeta\in \mathbb{H}^N$ and $t>0$.
	Then we shall calculate the derivatives of $\phi_R(\eta,t)$.
	Since $\supp F^*(s_R(\eta,t)) = \{(\eta,t) \in\mathbb{H}^N\times  [0,\infty): 1\le s_R(\eta,t)\le2\}$,
	by \eqref{eq:3.1} we may assume that there exists  $C>0$ such that  $t\in [0,\infty)$ and $\eta = (x,y,\tau)$ satisfy
	\begin{equation}
		\label{eq:3.3}
		|x|\le CR^\frac{1}{2}, \quad |y| \le CR^\frac{1}{2}, \quad |\tau| \le CR, \quad t\le CR. 
	\end{equation}
	By \eqref{eq:3.1} and \eqref{eq:3.3} we have
	\begin{equation}
		\label{eq:3.4}
		\begin{split}
			&|\partial_t \phi_R(\eta,t)| \lesi \frac{1}{R} F^*(s_R(\eta,t))^\frac{1}{p} 
			= \frac{1}{R} \phi_R^*(\eta,t)^\frac{1}{p},\\
			&|\partial^2_{x_j}\phi_R(\eta,t)| \lesi \frac{1}{R} F^*(s_R( \eta,t))^\frac{1}{p} 
			= \frac{1}{R} \phi_R^*(\eta,t)^\frac{1}{p},\\
			&|y_j\partial^2_{x_j,\tau}\phi_R (\eta,t)| \lesi  \frac{1}{R} F^*(s_R(\eta,t))^\frac{1}{p} 
			= \frac{1}{R} \phi_R^*(\eta,t)^\frac{1}{p},\\
			&|x_j\partial^2_{y_j,\tau}\phi_R(\eta,t)| \lesi  \frac{1}{R} F^*(s_R(\eta,t))^\frac{1}{p} 
			= \frac{1}{R} \phi_R^*(\eta,t)^\frac{1}{p},\\
			&|(|x|^2+|y|^2)\partial^2_{\tau}\phi_R(\eta,t)|   \lesi \frac{1}{R} F^*(s_R(\eta,t))^\frac{1}{p} 
			= \frac{1}{R} \phi_R^*(\eta,t)^\frac{1}{p}.
		\end{split}
	\end{equation}
	By \eqref{eq:Lap}, \eqref{eq:3.2}, and \eqref{eq:3.4} we see that
	\begin{equation}
		\label{eq:3.5}
		|(-\partial_t - \Delta_\mathbb{H}) \psi_R(\eta,t)| \lesi \frac{1}{R}  \psi_R^*(\eta,t)^\frac{1}{p}.
	\end{equation}
	
	Substituting $\varphi = \psi_R$ into \eqref{eq:weak}, by \eqref{eq:3.5} and H\"{o}lder's inequality, we obtain
	\begin{equation}
		\label{eq:3.6}
		\begin{split}
			0
			&\le\int_0^T \int_{\mathbb{H}^N} u(\eta,t)^p \psi_R(\eta,t)\, \di \eta\di t + \int_{\mathbb{H}^N}  \psi_R(\eta,0) \, \di\mu(\eta)\\
			&\le \int_0^R \int_{\mathbb{H}^N} u(\eta,t) |(-\partial_t-\Delta_\mathbb{H}) \psi_R(\eta,t) |\,\di\eta \di t\\
			&\lesi \frac{1}{R} \int_0^R \int_{\mathbb{H}^N} u( \eta,t)  \psi^*_R(\eta,t)^\frac{1}{p} \,\di\eta \di t\\
			&\le \frac{1}{R} \left( \int\int_{\supp \psi_R^*}\, \di\eta\di t\right)^{1-\frac{1}{p}} \left(\int_0^R\int_{\mathbb{H}^N}u(\eta,t)^p\psi_R^*(\eta,t) \di\eta\di t\right)^\frac{1}{p}
		\end{split}
	\end{equation}
	for all $R\in (2\rho^2,T/2)$. On the other hand, it follows from \eqref{eq:6.2} that 
	\begin{equation*}
		\begin{split}
			&\frac{1}{R} \left( \int\int_{\supp \psi_R^*}\, \di\eta\di t\right)^{1-\frac{1}{p}}
			\le\frac{1}{R} \left(\int_0^{\frac{\sqrt{2}}{3}R}\int_{B(\zeta, (2/9)^\frac{1}{4}\sqrt{R})}\, \di\eta\di t\right)^{1-\frac{1}{p}}
			 \lesi R^{\frac{Q}{2p}(p-1)-\frac{1}{p}},\\
	\text{and} \ \ 		&\int_{\mathbb{H}^N}  \psi_R(\eta,0) \, \di\mu(\eta) \ge \mu (B(\zeta,9^{-\frac{1}{4}}R^\frac{1}{2})) \ge
			\mu (B(\zeta, (2/3)^\frac{1}{2}\rho))  =: m_\rho
		\end{split}
	\end{equation*}
	for all  $R\in (2\rho^2,T/2)$.
	This, together with \eqref{eq:3.6}, implies that 
	\begin{equation}
		\label{eq:3.7}
		\begin{split}
			&m_\rho + \int_0^T \int_{\mathbb{H}^N} u(\eta,t)^p \psi_R(\eta,t)\, \di \eta\di t  \le  CR^{\frac{Q}{2p}(p-1)-\frac{1}{p}} \left(\int_0^R\int_{\mathbb{H}^N}u(\eta,t)^p\psi_R^*(\eta,t) \di\eta\di t\right)^\frac{1}{p}
		\end{split}
	\end{equation}
	for all $R\in (2\rho^2,T/2)$.
	Let $\epsilon>0$ be a sufficiently small positive constant. For any $R\in (2\rho^2,T/2)$, set
	\begin{equation}
		\label{eq:3.8}
		\begin{split}
		&z(r) := \int_0^R \int_{\mathbb{H}^N} u(\eta,t)^p \psi_r^* (\eta,t) \,\di\eta\di t,\\
		&Z(R) := \int_0^R z(r) \min\{r^{-1}, \epsilon^{-1}\} \di r.
		\end{split}
	\end{equation}
	Since $F^*$ is decreasing on $[1,\infty)$ and $\supp \xi^*\subset [1,2]$, for any $(\eta,t) \in  \mathbb{H}^N\times(0,T)$ with $9(\dist_\mathbb{H}(\eta,\zeta)^4 + t^2)\ge R^2$, we have
	\begin{equation}
		\label{eq:3.9}
		\begin{split}
			\int_0^R \psi_r^* (\eta,t) \min\{r^{-1},\epsilon^{-1}\} \,\di r 
			&= \int_0^R F^*\left(9\frac{\dist_\mathbb{H}(\eta,\zeta)^4 + t^2}{R^2}\right)r^{-1} \,\di r\\
			&=\frac{1}{2} \int_{9(\dist_\mathbb{H}(\eta,\zeta)^4 + t^2)/R^2} F^*(s) s^{-1} \, \di s\\
			&\le \frac{1}{2} F^*\left(9\frac{\dist_\mathbb{H}(\eta,\zeta)^4 + t^2}{R^2}\right) \int_1^2 s^{-1}\,\di s
			\le C' \psi_R^*(\eta,t).
		\end{split}
	\end{equation}
	Since $\psi_R^*(\eta,t) = 0$ if $9(\dist_\mathbb{H}(\eta,\zeta)^4 + t^2)< R^2$, by \eqref{eq:3.8} and \eqref{eq:3.9} we obtain
	\begin{equation}
		\label{eq:3.10}
		\begin{split}
			&\int_0^R \int_{\mathbb{H}^N} u(\eta,t)^p \psi_R(\eta,t) \, \di\eta \di t \\
			&\ge \int_0^R \int_{\mathbb{H}^N} u(\eta,t)^p \psi_R^*(\eta,t) \, \di\eta \di t \\
			&\ge C'^{-1} \int_0^R \int_{\mathbb{H}^N} u(\eta,t)^p \left(\int_0^R \psi_r^*(\eta,t) \min\{r^{-1}, \epsilon^{-1}\} \, \di r \right) \, \di \eta\di t\\
			& = C'^{-1}\int_0^R\int_0^R\int_{\mathbb{H}^N}  u(\eta,t)^p \psi_r^*(\eta,t) \min\{r^{-1}, \epsilon^{-1}\} \, \di\eta \di t \di r\\
			&= C'^{-1} Z(R).
		\end{split}
	\end{equation}
	Therefore, we deduce from \eqref{eq:3.7}, \eqref{eq:3.8}, and \eqref{eq:3.10} that
	\begin{equation*}
		m_\rho + C'^{-1} Z(R) \le C R^{\frac{Q}{2p}(p-1)-\frac{1}{p}}(\max\{R, \epsilon\} Z'(R))^\frac{1}{p}
	\end{equation*}
	for all  $R\in (2\rho^2,T/2)$. Therefore, we have
	\begin{equation}
		\label{eq:3.11}
		\begin{split}
		&\int_{Z(2\rho^2)}^{Z(T/2)} [m_\rho + C'^{-1} Z]^{-p} \,\di Z \ge C^{-1} \int_{2\rho^2}^{T/2} R^{-\frac{Q}{2}(p-1)+1} (\max\{R, \epsilon\})^{-1} \di R.
		\end{split}
	\end{equation}
	Since
	\begin{equation*}
		\begin{split}
			\int_{Z(2\rho^2)}^{Z(T/2)} [m_\rho + C'^{-1} Z]^{-p} \,\di Z 
			& \le \int_{Z(2\rho^2)}^\infty [m_\rho + C'^{-1} Z]^{-p} \,\di Z\\
			&  \le \frac{C}{p-1} (Z(2\rho^2) + m_\rho)^{-(p-1)}\\
			& \le \frac{C}{p-1}m_\rho^{-(p-1)}, 
		\end{split}
	\end{equation*}
	by \eqref{eq:3.11} we obtain
	\[
	\frac{C}{p-1}m_\rho^{-(p-1)} \ge C^{-1} \int_{2\rho^2}^{T/2}R^{-\frac{Q}{2}(p-1)+1}  (\max\{R, \epsilon\})^{-1} \di R.
	\]
	Letting $\epsilon\to 0^+$, we see that
	\begin{equation}
		\label{eq:3.12}
		\frac{C}{p-1}m_\rho^{-(p-1)} \ge C^{-1} \int_{2\rho^2}^{T/2} R^{-\frac{Q}{2}(p-1)} \di R
	\end{equation}
	for all $\rho\in (0, \sqrt{T}/4)$.
	
	
	Let $p=p_{2,Q}$. Then the inequality \eqref{eq:3.12}  yields
	\[
	\mu (B(\zeta, (2/3)^\frac{1}{2}\rho))  = m_\rho  \le C \left[\log\frac{T}{4\rho^2}\right]^{-\frac{Q}{2}}
	\le C\left[\log\left(e+\frac{\sqrt{T}}{\rho}\right)\right]^{-\frac{Q}{2}}
	\]
	for all $\zeta\in\mathbb{H}^N$ and $\rho\in (0, \sqrt{T}/8)$.
	
	On the other hand, let $p\neq p_{2,Q}$.
	Setting $T= 32\rho^2$, then the inequality \eqref{eq:3.12}  yields
		\[
	\mu (B(\zeta, (2/3)^\frac{1}{2}\rho))  = m_\rho  \le C \rho^{Q-\frac{2}{p-1}}
	\]
	for all $\zeta\in\mathbb{H}^N$ and $\rho\in (0, \sqrt{T}/8)$.
	
	Setting $\sigma:= 8\rho$, 
	similarly to the proof in the case of $\alpha\in (0,2)$, we obtain
	\[
\sup_{\eta\in \mathbb{H}^N}\mu(B(\eta, \sigma)) \leq 
\left\{
\begin{array}{ll}
\displaystyle{C\left[\log \left( e+\frac{\sqrt{T}}{\sigma} \right)\right]^{-\f{Q}{2}}} & \mbox{if} \quad p=p_{2,Q},\vspace{3pt}\\
C\sigma^{Q-\frac{2}{p-1}}& \mbox{if} \quad p\neq p_{2,Q},\vspace{3pt}\\
\end{array}
\right.
\]
	for all $\sigma\in (0,\sqrt{T})$.
	Thus, assertions (i)--(iii) in Theorem~\ref{main thm} follow and the proof is complete.
\end{proof}

\begin{proof}[Proof of Corollary~{\rm \ref{cor 1}}]
	Let $u$ be a solution of problem \eqref{eq:Fujita} with \eqref{eq:Fujitaini} in $\HH^N\times[0,T)$, where $T \in(0,\infty)$. Since for a.a.~$\tau \in (0, T)$, $u_\tau (\eta,t) := u(\eta,t + \tau)$ is a solution of problem \eqref{eq:Fujita} 
	 with $\mu = u(\tau )$ in $\HH^N\times[0, T - \tau )$, the estimate \eqref{eq1-cor1} is just  a direct consequence of Theorem~\ref{main thm} with $\sigma = (T-\tau)^{1/\alpha}$.	
%
	This completes our proof.
\end{proof}

At the end of this section, we give a proof of Theorem~\ref{mainthm 2}.
\begin{proof}
	[Proof of Theorem~{\rm \ref{mainthm 2}}]
	Fix $\eta\in \HH^N$. Let $u$ be a solution of \eqref{eq:Fujita} in $\HH^N\times(0,T)$, where $T\in(0,\infty)$. Let $\alpha \in(0,2]$ and 
	$0 < t < T$. For each $n\ge 1$, we have 
	\[
	\HH^N\backslash B(\eta, 2^nt^\f{1}{\alpha}) = \bigcup_{i\ge 0}B(\eta,2^{n+i+1}t^\f{1}{\alpha})\backslash B(\eta, 2^{n+i}t^\f{1}{\alpha}).
	\]
	By the   covering lemma in Lemma~\ref{lem-covering lemma} a family of balls $\mathcal B:=\{B_{k,i}:=B(\eta_{k,i},t^{1/\alpha}): i \in \mathbb N, k\in J_i  \}$ for some index set $J_i$ such that
	\begin{enumerate}
		\item[\rm (i)] $\displaystyle{\HH^N\backslash B(\eta, 2^nt^\f{1}{\alpha})\subset \bigcup_{k,i}B_{k,i}}$;
		
		\item[\rm (ii)] $\displaystyle{\dist_\mathbb{H}(\eta_{k,i},\eta)\sim 2^{n+i}t^\f{1}{\alpha}}$ for each $k,i\in \mathbb N$;
		
		\item[\rm (iii)] $\sharp J_i\le  C2^{Q(n+i)}$ for each $i\in \mathbb N$ and some universal constant $C>0$.
	\end{enumerate}
	This, together with  Theorem~\ref{main thm} and \eqref{A1}, implies that 
\begin{equation}\label{eq-1st limit proof main thm 2}
		\begin{aligned}
			&\operatorname*{ess\,sup}_{0 < \tau < t/2} \int_{\HH^N\setminus B(\eta, 2^nt^\f{1}{\alpha})} G_\alpha(\zeta^{-1}\circ \eta,t-\tau) u(\zeta,\tau) \, \di\zeta \\
			&\leq  \sum_{i=1}^{\infty} \sum_{k\in J_i}\operatorname*{ess\,sup}_{0 < \tau < t/2} \int_{B_{k,i}} G_\alpha(\zeta^{-1}\circ \eta,t-\tau) u(\zeta,\tau) \, \di\zeta \\
			&	\leq C \operatorname*{ess\,sup}_{0 < \tau < t/2} \sup_{\xi\in \HH^N} \int_{B(\xi, t^\f{1}{\alpha})} u( \zeta,\tau) \, \di\zeta
			\sum_{i=1}^{\infty} \sum_{k\in J_i} \sup_{0 < \tau < t/2}  (t - \tau)^{-\f{Q}{\alpha}} g_\alpha\left(\f{\dist_\mathbb{H}(\eta_{k,i},\eta)}{c_2(t - \tau)^{{1/\alpha}}} \right)\\
			&\lesi t^{\f{Q}{\alpha} - \frac{1}{p-1}}\sum_{i=1}^{\infty} \sum_{k\in J_i} t^{-\f{Q}{\alpha}}2^{-(Q+\alpha)(n+i)}\\
			&\lesi t^{- \frac{1}{p-1}}  \sum_{i=1}^{\infty} 2^{-(n+i)\alpha}\simeq  t^{- \frac{1}{p-1}}2^{-n\alpha}\to 0\quad \mbox{as}\quad  n\to \vc.\\
		\end{aligned}
	\end{equation}
	Similarly,  we have
	\begin{equation}\label{eq-2nd limit proof main thm 2}
		\begin{aligned}
			&\int_{\HH^N\setminus B(\eta, 2^nt^\f{1}{\alpha})} G_\alpha(\zeta^{-1}\circ \eta,t)\, \di\mu(\zeta) \\
			&\leq C \sup_{\xi\in \HH^N} \mu(B(\xi, t^\f{1}{\alpha}))
			\sum_{i=1}^{\infty} \sum_{k\in J_i} \sup_{0 < \tau < t/2}  t^{-\f{Q}{\alpha}} g_\alpha\left(\f{\dist_\mathbb{H}(\eta_{k,i},\eta)}{c_2t^{{1/\alpha}}} \right)\\
			&\lesi t^{- \frac{1}{p-1}}2^{-n\alpha} \to 0 \quad \mbox{as} \quad n\to \vc.
		\end{aligned}
	\end{equation}
	From these two above estimates \eqref{eq-1st limit proof main thm 2} and \eqref{eq-2nd limit proof main thm 2},  Theorem~\ref{main thm} and Lemma~\ref{lem 2} we obtain
	\begin{equation}\label{eq1-proof main thm 2}
	\begin{split}
		&\int_{\HH^N} G_\alpha(\zeta^{-1}\circ \eta,t-\tau) u( \zeta,\tau) \, \di\zeta  < \infty \quad\mbox{and}\quad\int_{\HH^N} G_\alpha(\zeta^{-1}\circ \eta,t) \, \di\mu(\zeta)  < \infty,
	\end{split}
	\end{equation}
	for a.a.~$\tau \in (0, t/2)$.
	Let $\eta_n \in C_c(\HH^N)$ be such that
	$0 \leq \eta_n \leq 1$ in $\HH^N$, $\eta_n = 1$ on $B(\eta, 2^nt^{1/\alpha})$, $\eta_n = 0$ outside $B(\eta, 2^{n+1}t^{1/\alpha})$. 
	It follows from \eqref{eq1-proof main thm 2} that
	\begin{equation}\label{eq2-proof main thm 2}
		\begin{aligned}
			&\left| \int_{\HH^N} G_\alpha(\zeta^{-1}\circ \eta,t-\tau) u(\zeta,\tau) \, \di\zeta  - \int_{\HH^N} G_\alpha(\zeta^{-1}\circ \eta,t) \, \di\mu(\zeta)  \right|\\
			&\leq \biggl| \int_{\HH^N} G_\alpha(\zeta^{-1}\circ \eta,t) u( \zeta,\tau) \eta_n(y) \, \di\zeta  - \int_{\HH^N} G_\alpha(\zeta^{-1}\circ \eta,t) \eta_n(y) \, \di\mu(\zeta)  \biggr|\\
			& \quad
			+ \left| \int_{\HH^N} [G_\alpha(\zeta^{-1}\circ \eta,t-\tau) - G_\alpha(\zeta^{-1}\circ \eta,t)] u(\zeta,\tau) \eta_n(y)  \, \di\zeta  \right|\\
			& \quad + \int_{\HH^N\setminus B(\eta, 2^nt^\f{1}{\alpha})} G_\alpha(\zeta^{-1}\circ \eta,t-\tau) u(\zeta,\tau) \di\zeta\\
			& \quad  + \int_{\HH^N\setminus B(\eta, 2^nt^\f{1}{\alpha})} G_\alpha(\zeta^{-1}\circ \eta,t)\,\di\mu(\zeta) 
		\end{aligned}
	\end{equation}
	for $n = 1, 2, \ldots $ and a.a.~$\tau \in (0, t/2)$. By Lemma~\ref{lem 3} and the fact that $G_\alpha( \cdot,t)$ is continuous, we see that
	\begin{equation}\label{eq3-proof main thm 2}
	\begin{split}
		&\operatorname*{ess\,lim}_{\tau \to 0^+} \biggl[ \int_{\HH^N} G_\alpha(\zeta^{-1}\circ \eta,t) u( \zeta,\tau) \eta_n(y)  \, \di\zeta- \int_{\HH^N} G_\alpha(\zeta^{-1}\circ \eta,t) \eta_n(y) \, \di\mu(\zeta)  \biggr] = 0. 
		\end{split}
	\end{equation}
	Furthermore, by Lemmas~\ref{lem 2}and \ref{lem-time derivaitive of heat kernel} we have
	\begin{equation}\label{eq4-proof main thm 2}
		\begin{aligned}
			&\operatorname*{ess\,limsup}_{\tau \to 0^+} \left| \int_{\HH^N} [G_\alpha(\zeta^{-1}\circ \eta,t-\tau) - G_\alpha(\zeta^{-1}\circ \eta,t)] u(\tau, \zeta) \eta_n(y) \,\di\zeta  \right|\\
			&
			\leq \sup_{\zeta \in B(\eta, 2^{n+1}t^\f{1}{\alpha}), s \in (t/2, t)} |\partial_s G_\alpha(\zeta, s)| \operatorname*{ess\,limsup}_{\tau \to 0^+} \left[ \tau \int_{B(\eta, 2^{n+1}t^\f{1}{\alpha})} u(\zeta,\tau) \,\di\zeta  \right]\\
			& \lesi t^{-\f{Q}{\alpha}-1} \operatorname*{ess\,limsup}_{\tau \to 0^+} \left[ \tau \int_{B(\eta, 2^{n+1}t^\f{1}{\alpha})} u(\zeta,\tau) \, \di\zeta  \right]= 0.
		\end{aligned}
	\end{equation}
	From \eqref{eq2-proof main thm 2}, \eqref{eq3-proof main thm 2}, and \eqref{eq4-proof main thm 2},
	\begin{equation*}
	\begin{split}
	&\operatorname*{ess\,limsup}_{\tau \to 0^+} \left| \int_{\HH^N} G_\alpha(\zeta^{-1}\circ \eta,t-\tau) u(\zeta,\tau)  \, \di\zeta  - \int_{\HH^N} G_\alpha(\zeta^{-1}\circ \eta,t) \, \di\zeta  \right|\\
	&\leq \operatorname*{ess\,sup}_{0 < \tau < t/2} \int_{\HH^N\setminus B(\eta,2^nt^\f{1}{\alpha})} G_\alpha(\zeta^{-1}\circ \eta,t-\tau) u(\zeta,\tau)  \, \di\zeta + \int_{\HH^N\setminus B(\eta,2^nt^\f{1}{\alpha})} G_\alpha(\zeta^{-1}\circ \eta,t) \, \di\zeta 
	\end{split}
	\end{equation*}
	for $n = 1, 2, \ldots$. 
	This, in combination with \eqref{eq-1st limit proof main thm 2} and \eqref{eq-2nd limit proof main thm 2}, implies that
	\[
	\operatorname*{ess\,lim}_{\tau \to 0^+} \left| \int_{\HH^N} G_\alpha( \zeta^{-1}\circ \eta,t-\tau) u( \zeta,\tau)  \, \di\zeta  - \int_{\HH^N} G_\alpha(\zeta^{-1}\circ \eta,t) \, \di\zeta  \right| = 0.
	\]
	This and \eqref{eq- defn sols of parabolic eqs} yield that $u$ is a solution of problem \eqref{eq:Fujita} with \eqref{eq:Fujitaini} in $ \HH^N\times[0,T)$. 	
	This completes our proof.
\end{proof}

\section{Sufficient conditions for the solvability.}
In this section we shall prove Theorems~\ref{mainthm 3}--\ref{Theorem:SC3}. 
\begin{proof}
[Proof of Theorem~{\rm \ref{mainthm 3}}] 
It suffices to consider the case of $T=1$. Indeed, for any solution $u$ of problem \eqref{eq:Fujita} with \eqref{eq:Fujitaini} in $\mathbb{H}^N\times[0,T)$,
 we see from \eqref{eq:G3} that $u_\lambda (\eta,t) := \lambda^\frac{\alpha}{p-1} u(\delta_\lambda(\eta),\lambda^\alpha t)$ with $\lambda = T^{1/\alpha}$ is a solution of problem problem \eqref{eq:Fujita} with \eqref{eq:Fujitaini}  in $\mathbb{H}^N\times[0,1)$.
	Set $w(\eta,t) := [e^{t\Lambda_\alpha}\mu](\eta)$. Then it follows from  \eqref{eq:G4} and Lemma~\ref{lem 1} that
	\[
	\begin{aligned}
	\mathcal{F}[w](t) &:= e^{t\Lambda_\alpha}\mu+ \int_0^t e^{(t-s)\Lambda_\alpha}(2w(s))^p\, \di s\\
	& \leq w(t) + 2^p w(t) \int_0^t \|w(s)\|_{L^\infty(\HH^N)}^{p-1} \, \di s\\
&\leq w(t) + C w(t) \int_0^t s^{- \frac{N(p-1)}{\theta}}  [ \sup_{\eta\in \HH^N} \mu  ( B  ( \eta, s^\f{1}{\alpha} ) )  ]^{p-1}  \, \di s\\
&\leq w(t)  + C w(t)  [ \sup_{\eta\in \HH^N} \mu  ( B  ( \eta, t^\f{1}{\alpha} )  )  ]^{p-1}\int_0^t s^{- \frac{Q(p-1)}{\alpha}}  \, \di s\\
&\leq w(t) + C w(t) [ \sup_{\eta\in \HH^N} \mu  ( B  ( \eta, T^\f{1}{\alpha}  )  )  ]^{p-1}\int_0^T s^{- \frac{Q(p-1)}{\alpha}} \,  \di s\\
\end{aligned}
	\]
	This, together with the assumption of  Theorem~\ref{mainthm 3} and $1 < p < p_{\alpha,Q}$, implies that
	\[
	\mathcal{F}[w](t) \leq  [1 + C \gamma^{p-1}_C ]w(t) 
	\]
	for all $0 \leq t < 1$. Therefore, taking a sufficiently small $\gamma_C > 0$ if necessary, we obtain $\mathcal{F}[w](t) \leq 2w(t) $ for $0 \leq t < 1$. This means that $2w(t) $ is a supersolution of  problem \eqref{eq:Fujita} with \eqref{eq:Fujitaini} in $ \mathbb{H}^N\times[0,T)$. Then the theorem follows from Lemma~\ref{lem - supersolutions imply solution}.
\end{proof}

\begin{proof}[Proof of Theorem~{\rm \ref{mainthm 4}}]
For the same reasons as in the proof of  Theorem~\ref{mainthm 3},
it suffices to consider the case of $T=1$.
Assume \eqref{eq:SC2}. We can assume, without loss of generality, that  $1 < \theta < p$. 
Indeed, if $\theta\ge p$, then, for any $1<\theta'<p$, we apply Jensen's inequality to obtain
\begin{equation*}
\begin{split}
\left[\sup_{\eta\in \HH^N}\fint_{B(\eta,\sigma)} \mu(\eta)^{\theta'} \, \di\eta  \right]^{\frac{1}{\theta'}}
\le \sup_{\eta\in \HH^N} \left[\fint_{B(\eta,\sigma)} \mu(\eta)^{\theta} \, \di\eta  \right]^{\frac{1}{\theta}}
\le\gamma_D \sigma^{-\frac{\alpha}{p-1}}
\end{split}
\end{equation*}
for all $0<\sigma<T^{1/\alpha}$.
Thus, \eqref{eq:SC2} holds with $\theta$ replaced by $\theta'$.

	 Let $1 < \theta < p$ and set $w(\eta,t)  := [e^{t\Lambda_\alpha}f^\theta] (\eta)^{1/\theta}$. It follows from \eqref{eq:G4}, \eqref{eq:G5} and Jensen's inequality  that
	\[
	\begin{aligned}
	\mathcal{F}[w](t) &:= e^{t\Lambda_\alpha}f(s) + \int_0^t e^{(t-s)\Lambda_\alpha}(2w(s))^p \, \di s\\
	& \leq w(t)  + 2^p[e^{t\Lambda_\alpha}f^\theta] \int_0^t \|e^{t\Lambda_\alpha}f^\theta\|_{L^\infty(\HH^N)}^{\f{p}{\theta} - 1} \, \di s\\
	&\leq w(t)  + Cw(t) \|e^{t\Lambda_\alpha}f^\theta\|_{L^\infty(\HH^N)}^{1 - \f{1}{\theta}} \int_0^t \|e^{t\Lambda_\alpha}f^\theta\|_{L^\infty(\HH^N)}^{\f{p}{\theta} - 1} \, \di s
	\end{aligned}
	\]
	for all $0 \leq t < 1$. Furthermore, by Lemma~\ref{lem 1} and  \eqref{eq:SC2}  we have
	\[
	\|e^{t\Lambda_\alpha}f^\theta\|_{L^\infty(\HH^N)} \leq C\gamma_D^\theta t^{-\f{\theta}{p-1}}.
	\]
	This implies that
	\[
	\mathcal{F}[w](t) \leq  [1 + C\gamma_D^{p-1}]w(t), \quad 0 \leq t < 1.
	\]
	Therefore, taking a sufficiently small $\gamma_D > 0$ if necessary, we obtain $\mathcal{F}[w](t) \leq 2w(t)$ for $0 \leq t < 1$. At this stage, arguing similarly to the proof of Theorem~\ref{mainthm 3} we complete the proof of Theorem~\ref{mainthm 4}.
\end{proof}

 \begin{proof}[Proof of Theorem~{\rm \ref{Theorem:SC3}}]
 For the same reasons as in the proof of  Theorem~\ref{mainthm 3},
it suffices to consider the case of $T=1$.	
 	Let $\beta>0$ and $\rho = \rho(s)$ be as in \eqref{eq:SC31}. Let $L\ge e$ be such that
 	\begin{itemize}
 		\item[(a)] $\Psi_{\beta,L}(s) := s [\log(L+s)]^\beta$ is positive and convex in $(0,\infty)$;
 		\item[(b)] $s^p/\Psi_{\beta,L}$ and $\Psi_{\beta,L}/s$ are monotone increasing in $(0,\infty)$.
 	\end{itemize}
 	Let $\mu$ be a nonnegative measurable function in $\mathbb{H}^N$ satisfying \eqref{eq:SC32}.
 	Simce $\Psi_\beta(\tau) \simeq \Psi_{\beta,L}$ for $\tau>0$, it follows that
 	\begin{equation}
 		\label{eq:5.1}
 		\Psi_{\beta,L}^{-1}\left[ \fint_{B(\zeta,\sigma)} \Psi_{\beta,L}( \mu(\eta)) \, \di\eta \right] \le C \gamma_E \rho(\sigma) 
 	\end{equation}
 	for all $\eta\in\mathbb{H}^N$ and $\sigma\in(0,1)$.
 	Set 
 	\[
 	w(\eta,t) :=  \Psi_{\beta,L}^{-1}[e^{t\Lambda_\alpha} \Psi_{\beta,L}( \mu(\eta))  ].
 	\]
 	By \eqref{eq:5.1} we apply Lemma~\ref{lem 1} to obtain
 	\begin{equation*}
 		\begin{split}
 			\|e^{t\Lambda_\alpha} \Psi_{\beta,L}(\mu)\|_{L^\infty(\mathbb{H}^N)} 
 			&= \|\Psi_{\beta,L}(w(t))\|_{L^\infty(\mathbb{H}^N)} \\
 			&\le C\Psi_{\beta,L}(C\gamma_E \rho(t^\f{1}{\alpha}))\\
 			& \le C\Psi_{\beta,L}(\gamma_E \rho(t^\f{1}{\alpha})), 
 		\end{split}
 	\end{equation*}
 	which implies that
 	\begin{equation}
 		\label{eq:5.2}
 		\|w(t)\|_{L^\infty(\mathbb{H}^N)} \le \Psi_{\beta,L}^{-1} [C \Psi_{\beta,L}(\gamma_E \rho (t^\f{1}{\alpha}))]
 	\end{equation}
 	for $t\in (0,1)$.
 	Define 
 	\[
 	\mathcal{F}[w](t) := e^{t\Lambda_\alpha}\mu + \int_0^t e^{(t-s)\Lambda_\alpha} (2w(s))^p \, \di s, \quad t>0.
 	\]
 	Then  by \eqref{eq:G4}, \eqref{eq:G5}, and Jensen's inequality,
 	\begin{equation}
 		\label{eq:5.3}
 		\begin{split}
 			\mathcal{F}[w](t)
 			&\le w(t) + 2^p \int_0^t e^{(t-s)\Lambda_\alpha} \left[\frac{w(s)^p}{e^{s\Lambda_\alpha}\Psi_{\beta,L}(\mu)} e^{s\Lambda_\alpha}\Psi_{\beta,L}(\mu)\right] \, \di s\\
 			& \le w(t) + 2^p  \left[\int_0^t \left\| \frac{w(s)^p}{e^{s\Lambda_\alpha}\Psi_{\beta,L}(\mu)}\right\|_{L^\infty(\mathbb{H}^N)} \, \di s  \right] e^{t\Lambda_\alpha}\Psi_{\beta,L}(\mu)\\
 			& \le w(t) + 2^p  \left[\int_0^t \left\| \frac{w(s)^p}{\Psi_{\beta,L}(w(s))}\right\|_{L^\infty(\mathbb{H}^N)} \, \di s  \right] \left\|\frac{\Psi_{\beta,L}(w(t))}{w(t)}\right\|_{L^\infty(\mathbb{H}^N)} w(t)\\
 		\end{split}
 	\end{equation}
 	for $t>0$. On the other hand, by property (b) and \eqref{eq:5.2} we see that
 	\begin{equation}
 		\label{eq:5.4}
		\begin{split}
 		\left\| \frac{w(s)^p}{\Psi_{\beta,L}(w(s))}\right\|_{L^\infty(\mathbb{H}^N)} 
 		&\le  \frac{\|w(s)\|_{L^\infty(\mathbb{H}^N)}^p}{\Psi_{\beta,L}(\|w(s)\|_{L^\infty(\mathbb{H}^N)})}\le  \frac{[\Psi_{\beta,L}^{-1} [C \Psi_{\beta,L}(\gamma_E \rho (t^\f{1}{\alpha}))]]^p}{C\Psi_{\beta,L} (\gamma_E\rho(s^\f{1}{\alpha}))}
		\end{split}
 	\end{equation}
 	for $s\in (0,1)$.
 	By \eqref{eq:SC31} we have
 	\[
 	\Psi_{\beta,L} (\gamma_E\rho(s^\f{1}{\alpha})) = \gamma_E\rho(s^\f{1}{\alpha}) [\log(L + \gamma_E\rho(s^\f{1}{\alpha}))]^\beta \simeq \gamma_E s^{-\frac{Q}{\alpha}} \left[\log\left(e+ \frac{1}{s}\right)\right]^{-\frac{Q}{\alpha}+\beta}
 	\]
 	for all $s\in (0,1)$.
 	Since $\Psi_{\beta,L}^{-1} (\tau) \simeq \tau [\log(e+\tau)]^{-\beta}$ for all $\tau>0$, it follows that
 	\[
 	\Psi_{\beta,L}^{-1} (C \Psi_{\beta,L} (\gamma_E \rho(t^\f{1}{\alpha}))) \simeq \gamma s^{-\frac{Q}{\alpha}} \left[\log\left(e+\frac{1}{s}\right)\right]^{-\frac{Q}{\alpha}}
 	\]
 	for all $s\in (0,1)$.
 	These together with \eqref{eq:5.4} imply that 
 	\begin{equation}
 		\label{eq:5.5}
 		\left\| \frac{w(s)^p}{\Psi_{\beta,L}(w(s))}\right\|_{L^\infty(\mathbb{H}^N)}  \le C\gamma_E^\frac{\alpha}{Q} s^{-1} \left[\log\left(e+\frac{1}{s}\right)\right]^{-1-\beta}
 	\end{equation}
 	for all $s\in (0,1)$.
 	Similarly, by \eqref{eq:5.2} and property (b) we have
 	\begin{equation}
 		\label{eq:5.6}
 		\left\|\frac{\Psi_{\beta,L}(w(t))}{w(t)}\right\|_{L^\infty(\mathbb{H}^N)} 
 		\le \frac{C\Psi_{\beta,L}(\gamma_E \rho(t^\f{1}{\alpha}))}{\Psi_{\beta,L}^{-1}(C \Psi_{\beta,L}(\gamma_E \rho(t^\f{1}{\alpha})))} \le C \left[\log\left(e+\frac{1}{t}\right)\right]^\beta
 	\end{equation}
 	for all $t\in(0,1)$.
 	By \eqref{eq:5.5} and \eqref{eq:5.6} we obtain
 	\begin{equation}
 		\label{eq:5.7}
 		\begin{split}
 			&\left[\int_0^t \left\| \frac{w(s)^p}{\Psi_{\beta,L}(w(s))}\right\|_{L^\infty(\mathbb{H}^N)} \, \di s  \right] 
 			\left\|\frac{\Psi_{\beta,L}(w(t))}{w(t)}\right\|_{L^\infty(\mathbb{H}^N)} \\
 			&\le C\gamma_E^\frac{\alpha}{Q} \left[\log\left(e+\frac{1}{t}\right)\right]^\beta \int_0^t s^{-1}\left[\log\left(e+\frac{1}{s}\right)\right]^{-1-\beta}\, \di s \le C\gamma_E^\frac{\alpha}{Q}
 		\end{split}
 	\end{equation}
 	for all $t\in(0,1)$.
 	Therefore, taking a sufficiently small $\gamma_E > 0$ if necessary, we deduce from \eqref{eq:5.3} and \eqref{eq:5.7} that $\mathcal{F}[w](t) \leq 2w(t)$ for $t\in (0,1)$. At this stage, arguing similarly to the proof of Theorem~\ref{mainthm 4} we complete the proof of Theorem~\ref{Theorem:SC3}.
 \end{proof}
 
 Finally, we give a proof of Corollary~\ref{Cor:1.2}.
 \begin{proof}[Proof of Corollary~{\rm \ref{Cor:1.2}}]
 Assume that $\mu= \gamma \Phi_\alpha + C_\alpha$ for some $\gamma>0$ and $C_\alpha\ge 0$.
 Let  $p>p_{\alpha, Q}$ and   fix $1<\theta<p$ so  that it satisfies
 \[
 \frac{\alpha \theta}{p-1} < Q.
 \]
  It follows from \eqref{eq:6.1} and \eqref{eq:6.2} that
  \begin{equation*}
  \begin{split}
  \sup_{\zeta\in\mathbb{H}^N} \left[\fint_{B(\zeta,\sigma)} \mu(\eta) \, \di \eta\right]^{\frac{1}{\theta}}
  = \left[\fint_{B(0,\sigma)} \mu(\eta) \, \di \eta\right]^{\frac{1}{\theta}}\lesi \gamma \sigma^{-\frac{\alpha}{p-1}} + C_\alpha
 \end{split}
  \end{equation*}
  for $\sigma>0$.
  Then taking   sufficiently small $\gamma>0$ and $T>0$ if necessary, we see that \eqref{eq:SC2} holds for all $0<\sigma<T^{1/\alpha}$. This implies that problem \eqref{eq:Fujita} with \eqref{eq:Fujitaini} possesses a solution in $\mathbb{H}^N\times[0,T)$.
  If $C_\alpha=0$, we can take $T=\infty$. 
 Let $p=p_{\alpha,Q}$ and $\beta>0$. By similar calculations, we see that
 \begin{equation*}
 \begin{split}
		\sup_{\zeta \in \mathbb{H}^N}  \Psi_\beta^{-1}\left[ \fint_{B(\zeta,\sigma)} \Psi_\beta(T^\frac{1}{p-1} \mu(\eta)) \, \di\eta \right]  
		&= \Psi_\beta^{-1}\left[ \fint_{B(0,\sigma)} \Psi_\beta(T^\frac{1}{p-1} \mu(\eta)) \, \di\eta \right] \\
		&\lesi \gamma \rho(\sigma T^{-\frac{1}{\alpha}}) +C_\alpha T^\frac{1}{p-1}
\end{split}
\end{equation*}
for all $\sigma>0$, where $\Psi_\beta$ and $\rho$ is as in \eqref{eq:SC31}.
  Then taking   sufficiently small $\gamma>0$ and $T>0$ if necessary, we see that \eqref{eq:SC32} holds for all $0<\sigma<T^{1/\alpha}$. This implies that problem \eqref{eq:Fujita} with \eqref{eq:Fujitaini} possesses a solution in $\mathbb{H}^N\times[0,T)$.
 
 On the other hand, it follows from \eqref{eq:6.1} and \eqref{eq:6.2} that
 \begin{equation*}
 \begin{split}
 \int_{B(0,\sigma)} \mu(\zeta) \, \di \zeta \gtrsim 
 \left\{
	\begin{aligned}
	& \gamma \left[\log(e+\sigma^{-1})\right]^{-\frac{Q}{\alpha}} + C_\alpha^Q \quad && \mbox{if} \quad p=p_{\alpha,Q}, \\
	& \gamma \sigma^{Q-\frac{\alpha}{p-1}} +C_\alpha^Q\quad && \mbox{if} \quad p>p_{\alpha,Q}.\\	
	\end{aligned}
	\right.
 \end{split}
 \end{equation*}
 Then we see that if $\gamma>0$ is sufficiently large, assertions (ii) and (iii) in Theorem~\ref{main thm} do not hold for all $\sigma>0$.
 This implies that problem \eqref{eq:Fujita} with \eqref{eq:Fujitaini} possesses no local-in-time solutions.
 Thus, the proof is complete.
 \end{proof}
 
\section{Application.}

Since the minimal solution is unique, we can define the life span $T(\mu)$ as the maximal existence time  of the minimal solution of problem \eqref{eq:Fujita} with \eqref{eq:Fujitaini}.

For \eqref{eq:fjt} and in the case of $\alpha=2$, Lee--Ni \cite{LN92} obtained sharp estimates of $T(\lambda\phi)$  as $\lambda\to0^+$ by use of the behavior of $\phi$ at the space infinity. 
Subsequently,  the second author of this paper and Ishige \cite{HI18} obtained a generalization to the case of $\alpha\in (0,2]$.
Recently, Georgiev--Palmieri \cite{GP21} obtained a generalization of \cite{LN92} to the Heisenberg group $\mathbb{H}^N$ in the case of $\alpha=2$.
In some cases, however, sharp estimates have not yet been obtained.

In this section, as an application of our theorems, we show that similar estimates of $T(\lambda\phi)$ as in \cites{LN92, HI18} in the Heisenberg group $\mathbb{H}^N$. 
Theorems~\ref{Theorem F} and \ref{Theorem G} are generalizations of \cite[Theorem~3.15 and Theorem~3.21]{LN92}, respectively.
At the end of this section, summaries of these theorems and previous study \cite{GP21} are given.

\begin{thm}
\label{Theorem F}
Let $N\ge1$ and $p>1$.
Let $A>0$ and $\phi$ be a nonnegative measurable function in $\mathbb{H}^N$ such that
\begin{equation*}
(1+ |\eta|_{\mathbb{H}^N})^{-A} \le \phi(\eta)  \in L^\infty(\mathbb{H}^N)
\end{equation*}
for a.a.~$\eta\in \mathbb{H}^N$.
\begin{itemize}
\item[(i)] Let $p=p_{\alpha, Q}$ and $A\ge \alpha /(p-1) =Q$. Then there exists a  constant $C>0$ such that
\[
\log T(\lambda\phi) \le \left\{
	\begin{aligned}
	&C \lambda^{-(p-1)}\quad && \mbox{if} \quad A>Q, \\
	&C \lambda^{-\frac{p-1}{p}}\quad && \mbox{if} \quad A=Q, \\	
	\end{aligned}
	\right.
\]
for all sufficiently small $\lambda>0$.
\item[(ii)] Let $1<p<p_{\alpha, Q}$ or $A<\alpha/(p-1)$. Then there exists a  constant $C'>0$ such that
\[
T(\lambda\phi) \le \left\{
	\begin{aligned}
	&C' \lambda^{-\left( \frac{1}{p-1} -\frac{1}{\alpha} \min\{A,Q\}\right)^{-1}}\quad && \mbox{if} \quad A\neq Q, \\
	&C' \left( \frac{\lambda^{-1}}{\log(\lambda^{-1})}\right)^{\left(\frac{1}{p-1}-\frac{Q}{\alpha}\right)^{-1}}\quad && \mbox{if} \quad A=Q, \\	
	\end{aligned}
	\right.
\]
for all sufficiently small $\lambda>0$.
\end{itemize}
\end{thm}

For the simplicity of notation, we denote $T_\lambda := T(\lambda \phi)$.
We give a proof of Theorem~\ref{Theorem F}.

\begin{proof}[Proof of Theorem~{\rm \ref{Theorem F}}]
Since $\phi\in L^\infty(\mathbb{H}^N)$, by Theorem~\ref{mainthm 4} we have
\[
T_\lambda \gtrsim   \lambda^{-(p-1)}
\]
for all sufficiently small $\lambda>0$.
This implies that $\lim_{\lambda\to 0^+} T_\lambda = \infty$.
So, we can assume without loss of generality that  $T_\lambda>0$ is sufficiently large.

We apply Theorem~\ref{main thm} to prove Theorem~\ref{Theorem F}
and
assume that problem \eqref{eq:Fujita} with $\mu = \lambda \phi$ possesses a solution in $\mathbb{H}^N\times[0,T_\lambda)$.
For any $p>1$, we see that
\begin{equation}
\label{eq:6.3}
\begin{split}
\int_{B(0,\sigma)} \lambda \phi(\eta) \, \di\eta
& \ge \lambda \int_{B(0,\sigma)} (1+ |\eta|_{\mathbb{H}^N})^{-A} \, \di \eta\\
&\gtrsim  \left\{
	\begin{aligned}
	&\lambda \quad && \mbox{if} \quad \sigma>1,  A> Q, \\
	&\lambda \log (e+ \sigma) \quad && \mbox{if} \quad \sigma>1, A=Q, \\
	&\lambda  \sigma^{Q-A}\quad && \mbox{if} \quad \sigma>1,  0<A<Q, \\	
	\end{aligned}
	\right.
\end{split}
\end{equation}
for all $\sigma>1$ and sufficiently small $\lambda>0$.
In the case of $p=p_{\alpha,Q}$, it follows from assertion (ii) in Theorem~\ref{main thm} that
\[
\int_{B(0,\sigma)} \lambda \phi(\eta) \, \di\eta \le \gamma_A\left[\log\left(e+ \frac{T_\lambda^{1/\alpha}}{\sigma}\right)\right]^{-\frac{Q}{\alpha}}
\]
for all $0<\sigma < T_\lambda^{1/\alpha}$ and sufficiently small $\lambda>0$.
This implies that
\begin{equation}
\label{eq:6.4}
\int_{B(0,T_\lambda^\frac{1}{2\alpha})} \lambda \phi(\eta) \, \di\eta \lesi  
\gamma_A [\log T_\lambda]^{-\frac{Q}{\alpha}},
\end{equation}
\begin{equation}
\label{eq:6.44}
\int_{B(0,T_\lambda^\frac{1}{\alpha})} \lambda \phi(\eta) \, \di\eta \lesi  
\gamma_A,
\end{equation}
for all sufficiently small $\lambda>0$.
By \eqref{eq:6.3} and \eqref{eq:6.4} with $\sigma = T_\lambda^{1/2\alpha}$ we obtain
assertion (i).
Furthermore, by \eqref{eq:6.3} and \eqref{eq:6.44} with $\sigma = T_\lambda^{1/\alpha}$ we obtain assertion (ii)
in the case where $p=p_{\alpha,Q}$ and $A<1/(p-1)$.

We prove assertion (ii) in the case of $1<p<p_{\alpha,Q}$.
By assertion (i) in Theorem~\ref{main thm} we see that
\begin{equation}
\label{eq:6.444}
\int_{B(0,T_\lambda^\f{1}{\alpha})} \lambda \phi (\eta) \, \di \eta \le \gamma_A T_\lambda^{\frac{Q}{\alpha} - \frac{1}{p-1}}.
\end{equation}
By \eqref{eq:6.3} and  \eqref{eq:6.444}, we obtain assertion (ii) in the case of $1<p<p_{\alpha,Q}$. Similarly, we obtain assertion (ii) in the case of $p>p_{\alpha,Q}$. Thus, the proof is complete.
\end{proof}

\begin{thm}
\label{Theorem G}
Let $N\ge1$ and $p>1$.
Let $A>0$ and $\phi$ be a nonnegative measurable function in $\mathbb{H}^N$ such that
\begin{equation*}
0\le \phi(\eta) \le (1+ |\eta|_{\mathbb{H}^N})^{-A}
\end{equation*}
for a.a.~$\eta\in \mathbb{H}^N$.
\begin{itemize}
\item[(i)] Let $p=p_{\alpha, Q}$ and $A\ge \alpha /(p-1) =Q$. Then there exists a  constant $C>0$ such that
\[
\log T(\lambda\phi) \ge \left\{
	\begin{aligned}
	&C \lambda^{-(p-1)}\quad && \mbox{if} \quad A>Q, \\
	&C \lambda^{-\frac{p-1}{p}}\quad && \mbox{if} \quad A=Q, \\	
	\end{aligned}
	\right.
\]
for all sufficiently small $\lambda>0$.
\item[(ii)] Let $1<p<p_{\alpha, Q}$ or $A<\alpha/(p-1)$. Then there exists a positive constant $C'>0$ such that
\[
T(\lambda\phi) \ge \left\{
	\begin{aligned}
	&C' \lambda^{-\left( \frac{1}{p-1} -\frac{1}{\alpha} \min\{A,Q\}\right)^{-1}}\quad && \mbox{if} \quad A\neq Q, \\
	&C' \left( \frac{\lambda^{-1}}{\log(\lambda^{-1})}\right)^{\left(\frac{1}{p-1}-\frac{Q}{\alpha}\right)^{-1}}\quad && \mbox{if} \quad A=Q, \\	
	\end{aligned}
	\right.
\]
for all sufficiently small $\lambda>0$.
\end{itemize}
\end{thm}

\begin{proof}
We apply Theorem~\ref{Theorem:SC3} to prove assertion (i).
Let $p=p_{\alpha, Q}$ and set 
\begin{equation*}
		\Psi(s) := s[\log(e+s)]^\frac{Q}{\alpha}, \qquad
		\rho(s) := s^{-Q} \left[\log\left(e+\frac{1}{s}\right)\right]^{-\frac{Q}{\alpha}},
\end{equation*}
for $s>0$
(see \eqref{eq:SC31}).
For any $\lambda\in(0,1)$ and $\epsilon \in(0,1)$, set 
\begin{equation*}
\overline{T}_\lambda := \left\{
	\begin{aligned}
	&\exp(\epsilon \lambda^{-(p-1)}),\quad && \mbox{if} \quad A>Q, \\
	&\exp(\epsilon \lambda^{-\frac{p-1}{p}}),\quad && \mbox{if} \quad A=Q.\\	
	\end{aligned}
	\right.
\end{equation*}
We shall prove that problem \eqref{eq:Fujita} with $\mu = \lambda \phi$ possesses a solution in $\mathbb{H}^N\times[0,\overline{T}_\lambda)$.
Let $L\ge e$ be  such that
\[
s [\log (L+s)]^{-\frac{Q}{\alpha}} \quad \mbox{is increasing in} \,\, [0,\infty).
\]
Then we see that $\Psi(s) \sim s [\log (L+s)]^{{Q/\alpha}}$ and  $\Psi^{-1}(s) \sim s [\log (e+s)]^{-{Q/\alpha}} \sim s [\log (L+s)]^{-{Q/\alpha}}$
for all $s>0$.

We consider the case of $A>Q$. By \eqref{eq:6.1} and \eqref{eq:6.2}, we have
\begin{equation}
\label{eq:6.5}
\begin{split}
& \sup_{\zeta\in\mathbb{H}^N}  \Psi^{-1} \left[\fint_{B(\zeta,\sigma)} \Psi\left(\overline{T}_\lambda^\frac{1}{p-1} \lambda \phi(\eta) \right)\, \di \eta\right]\\
&\le \Psi^{-1} \left[\fint_{B(0,\sigma)} \Psi\left(\overline{T}_\lambda^\frac{Q}{\alpha} \lambda (1+|\eta|_{\mathbb{H}^N})^{-A}\right)\, \di \eta\right]\\
&\le \Psi^{-1} \left[C\sigma^{-Q}\int_0^{\sigma^2}\Psi\left(\overline{T}_\lambda^\frac{Q}{\alpha} \lambda (1+\sqrt{r})^{-A}\right)  r^N \, \di r\right]\\
\end{split}
\end{equation}
for all $\sigma>0$.
Since
\begin{equation}
\label{eq:6.6}
\log \left[L + \overline{T}_\lambda^\frac{Q}{\alpha} \lambda (1+\sqrt{r})^{-A}\right]
\le  \log (C \overline{T}_\lambda^\frac{Q}{\alpha})
\lesi\epsilon \lambda^{-\frac{\alpha}{Q}}
\end{equation}
for sufficiently small $\lambda>0$, we have
\begin{equation}
\label{eq:6.7}
\begin{split}
&\sigma^{-Q}\int_0^{\sigma^2}\Psi\left(\overline{T}_\lambda^\frac{Q}{\alpha} \lambda (1+\sqrt{r})^{-A}\right)  r^N \, \di r\\
&\lesi  \epsilon^\frac{Q}{\alpha} \overline{T}_\lambda^\frac{Q}{\alpha} \sigma^{-Q}\int_0^{\sigma^2} (1+\sqrt{r})^{-A}  r^N \, \di r
\lesi  \epsilon^\frac{Q}{\alpha} \overline{T}_\lambda^\frac{Q}{\alpha} \sigma^{-Q}
\end{split}
\end{equation}
for all $0<\sigma <  \overline{T}_\lambda^{1/\alpha}$ and sufficiently small $\lambda>0$.
This together with \eqref{eq:6.5} implies that
\begin{equation*}
\begin{split}
&\sup_{\zeta\in\mathbb{H}^N}  \Psi^{-1} \left[\fint_{B(\zeta,\sigma)} \Psi\left(\overline{T}_\lambda^\frac{1}{p-1} \lambda \phi(\eta) \right)\, \di \eta\right]\\
&\le  \Psi^{-1}( C \epsilon^\frac{Q}{\alpha} \overline{T}_\lambda^\frac{Q}{\alpha} \sigma^{-Q})
\lesi  \epsilon^\frac{Q}{\alpha} \overline{T}_\lambda^\frac{Q}{\alpha} \sigma^{-Q} \left[\log\left(L +  C \epsilon^\frac{Q}{\alpha} \overline{T}_\lambda^\frac{Q}{\alpha} \sigma^{-Q}\right)\right]^{-\frac{Q}{\alpha}}\\
&\lesi \epsilon^\frac{Q}{\alpha}\overline{T}_\lambda^\frac{Q}{\alpha} \sigma^{-Q} 
\left[\log\left(L +  \frac{\overline{T}_\lambda^{1/\alpha}}{\sigma}\right)\right]^{-\frac{Q}{\alpha}}
=   \epsilon^\frac{Q}{\alpha}  \rho(\sigma \overline{T}_\lambda^{-\frac{1}{\alpha}})\\
\end{split}
\end{equation*}
for all $0<\sigma <  \overline{T}_\lambda^{1/\alpha}$ and sufficiently small $\lambda>0$.
Therefore, taking a sufficiently small $\epsilon>0$ if necessary, we apply  Theorem~\ref{Theorem:SC3}
to see that  problem \eqref{eq:Fujita} with $\mu = \lambda \phi$ possesses a solution in $\mathbb{H}^N\times[0,T_\lambda)$ and 
\[
T(\lambda\phi) \ge \overline{T}_\lambda =\exp(\epsilon \lambda^{-(p-1)})
\]
for all sufficiently small $\lambda>0$.

We consider the case of $A=Q$. 
Similarly to \eqref{eq:6.6} and \eqref{eq:6.7}, we have
\begin{equation*}
\begin{split}
\sigma^{-Q}\int_0^{\sigma^2}\Psi\left(\overline{T}_\lambda^\frac{Q}{\alpha} \lambda (1+\sqrt{r})^{-A}\right)  r^N \, \di r
&\lesi \lambda \overline{T}_\lambda^\frac{Q}{\alpha} [\log \overline{T}_\lambda ]^\frac{Q}{\alpha}\sigma^{-Q}\int_0^{\sigma^2} (1+\sqrt{r})^{-A}  r^N \, \di r\\
&\lesi \lambda \overline{T}_\lambda^\frac{Q}{\alpha} [\log \overline{T}_\lambda ]^{\frac{Q}{\alpha}+1}\sigma^{-Q}
\lesi \epsilon^{\frac{Q+\alpha}{\alpha}} \overline{T}_\lambda^\frac{Q}{\alpha}\sigma^{-Q}
\end{split}
\end{equation*}
for all $0<\sigma <  \overline{T}_\lambda^{1/\alpha}$ and sufficiently small $\lambda>0$.
Then we apply the same argument as in the case of $A>Q$ to see that
\[
T(\lambda \phi) \ge \overline{T}_\lambda = \exp(\epsilon \lambda^{-\frac{p-1}{p}})
\]
for all sufficiently small $\lambda >0$. Thus, assertion (i) follows.

We shall prove assertion (ii) in the case where $p\ge p_{\alpha,Q}$ and $A<\alpha/(p-1)$.
It follows that $A<\alpha/(p-1) \le Q$.
For $\lambda\in(0,1)$ and $\epsilon\in(0,1)$, set
\[
\tilde{T}_\lambda := \epsilon \lambda^{-\left(\frac{1}{p-1}-\frac{A}{\alpha}\right)^{-1}}.
\]
Let $\theta>1$ be  such that $A\theta < Q$.
Then by \eqref{eq:6.1} and \eqref{eq:6.2} we have
\begin{equation}
\label{eq:6.8}
\left(\fint_{B(\eta,\sigma)} (\lambda \phi(\zeta))^\theta \, \di \zeta\right)^\frac{1}{\theta}
\le \lambda \left(\fint_{B(0,\sigma)} (1+|\zeta|_{\mathbb{H}^N})^{-A\theta}\right)^\frac{1}{\theta}
\lesi\lambda \sigma^{-A}
\end{equation}
for all $\zeta\in \mathbb{H}^N$ and $0<\sigma < \tilde{T}_\lambda^{1/\alpha}$.
On the other hand, it follows that
\begin{equation}
\label{eq:6.9}
\lambda\sigma^{-A} = \sigma^{-\frac{\alpha}{p-1}} \cdot \lambda \sigma^{\frac{\alpha}{p-1}-A}
\le \sigma^{-\frac{\alpha}{p-1}} \cdot \lambda \tilde{T}_\lambda^{\frac{1}{p-1}-\frac{A}{\alpha}}
=\epsilon^{\frac{1}{p-1}- \frac{A}{\alpha}} \sigma^{-\frac{\alpha}{p-1}}
\end{equation}
for all $0<\sigma < \tilde{T}_\lambda^{1/\alpha}$.
Taking a sufficiently small $\epsilon>0$ if necessary, by \eqref{eq:6.8} and \eqref{eq:6.9} we obtain \eqref{eq:SC2}
for all $0<\sigma < \tilde{T}_\lambda^{1/\alpha}$.
Then it follows from Theorem~\ref{mainthm 4} that  problem \eqref{eq:Fujita} with $\mu = \lambda \phi$ possesses a solution in $\mathbb{H}^N\times[0,\tilde{T}_\lambda)$
and
\[
T_\lambda \ge \tilde{T}_\lambda = \epsilon \lambda^{-\left(\frac{1}{p-1}-\frac{A}{\alpha}\right)^{-1}}.
\]
Thus, assertion (ii) in the case where $p\ge p_{\alpha,Q}$ and $A<\alpha/(p-1)$ follows.

It remains to prove assertion (ii) in the case of $1<p<p_{\alpha,Q}$. For $\lambda\in(0,1)$ and  $\epsilon\in(0,1)$, set
\[
\hat{T}_\lambda := \left\{
	\begin{aligned}
	&\epsilon \lambda^{-\left( \frac{1}{p-1} -\frac{1}{\alpha} \min\{A,Q\}\right)^{-1}}\quad && \mbox{if} \quad A\neq Q, \\
	&\epsilon \left( \frac{\lambda^{-1}}{\log(\lambda^{-1})}\right)^{\left(\frac{1}{p-1}-\frac{Q}{\alpha}\right)^{-1}}\quad && \mbox{if} \quad A=Q. \\	
	\end{aligned}
	\right.
\]
It follows from \eqref{eq:6.1} that
\[
\sup_{\zeta\in \mathbb{H}^N} \int_{B(\zeta,\hat{T}_\lambda^\frac{1}{\alpha} )} \lambda \phi(\eta) \, \di\eta
\lesi\left\{
	\begin{aligned}
	&\lambda \quad && \mbox{if} \quad A> Q, \\
	&\lambda \log (e+ \hat{T}_\lambda^\frac{1}{\alpha}) \quad && \mbox{if} \quad A=Q, \\
	&\lambda  \hat{T}_\lambda^\frac{Q-A}{\alpha}\quad && \mbox{if} \quad 0<A<Q. \\	
	\end{aligned}
	\right.
\]
Then taking a sufficiently small $\epsilon>0$ if necessary, we obtain 
\[
\sup_{\zeta\in \mathbb{H}^N} \int_{B(\zeta,\hat{T}_\lambda^\frac{1}{\alpha} )} \lambda \phi(\eta) \, \di\eta
\le \gamma_C \hat{T}_\lambda^{\frac{Q}{\alpha}-\frac{1}{p-1}}
\]
for all sufficiently small $\lambda>0$, where $\gamma_C>0$ is as in Theorem~\ref{mainthm 3}.
Then it follows from Theorem~\ref{mainthm 3} that  problem \eqref{eq:Fujita} with $\mu = \lambda \phi$ possesses a solution in $\mathbb{H}^N\times[0,\tilde{T}_\lambda)$
and
\[
T_\lambda \ge \hat{T}_\lambda = \left\{
	\begin{aligned}
	&\epsilon \lambda^{-\left( \frac{1}{p-1} -\frac{1}{\alpha} \min\{A,Q\}\right)^{-1}}\quad && \mbox{if} \quad A\neq Q, \\
	&\epsilon \left( \frac{\lambda^{-1}}{\log(\lambda^{-1})}\right)^{\left(\frac{1}{p-1}-\frac{Q}{\alpha}\right)^{-1}}\quad && \mbox{if} \quad A=Q. \\	
	\end{aligned}
	\right.
\]
Thus, the proof is complete.
\end{proof}

At the end of this section, we describe  summaries of Theorems~\ref{Theorem F} and \ref{Theorem G} in tables.
The following  tables show the behavior of the life span $T_\lambda:= T(\lambda \phi)$ as $\lambda \to0^+$, where 
\[
\phi(\eta) := (1+ |\eta|_{\mathbb{H}^N})^{-A}\quad \mbox{and} \quad A>0.
\]
If it is marked with {\dag}, it is already shown in \cite{GP21} in the case of $\alpha=2$.
\begin{table}[ht]
 \begin{center}
   \caption{The behavior of $T_\lambda$ in the case of $A \neq Q$ (as $\lambda\to0^+$).}
  \begin{tabular}{|l||r|r|r|}
  \hline
    \backslashbox{$p$}{$A$}		&$A<\frac{\alpha}{p-1}$ 	& $A=\frac{\alpha}{p-1}$ 	& $A>\frac{\alpha}{p-1}$ \\	\hline\hline
    $p<p_{\alpha,Q}$ 			& $T_\lambda \sim 	\lambda^{-\left( \frac{1}{p-1} -\frac{1}{\alpha} \min\{A,Q\}\right)^{-1}}$				&$T_\lambda \sim 	\lambda^{-\left( \frac{1}{p-1} -\frac{Q}{\alpha}\right)^{-1}}$ 												&{\dag} $T_\lambda \sim 	\lambda^{-\left( \frac{1}{p-1} -\frac{Q}{\alpha}\right)^{-1}}$\\\hline
    $p= p_{\alpha,Q}$ 			&$T_\lambda \sim 	\lambda^{-\left( \frac{1}{p-1} -\frac{A}{\alpha}\right)^{-1}}$							& ($A=Q$, see Table 2)					
    								&{\dag} $\log T_\lambda \sim \lambda^{-(p-1)}$\\	\hline
    $p>p_{\alpha,Q}$ 			& $T_\lambda \sim 	\lambda^{-\left( \frac{1}{p-1} -\frac{A}{\alpha}\right)^{-1}}$							&{\dag} $T_\lambda =\infty$				
    								&$T_\lambda =\infty$\\\hline
  \end{tabular}
 \end{center}
\end{table}

\begin{table}[ht]
 \begin{center}
   \caption{The behavior of $T_\lambda$ in the case of $A = Q$ (as $\lambda\to0^+$).}
  \begin{tabular}{|l||r|r|r|}
  \hline
    \backslashbox{$p$}{$A$}			& $A=Q$ 	 \\	\hline\hline
    $p<p_{\alpha,Q}$ 				&$T_\lambda \sim 	\left( \frac{\lambda^{-1}}{\log(\lambda^{-1})}\right)^{\left(\frac{1}{p-1}-\frac{Q}{\alpha}\right)^{-1}}$\\\hline
    $p= p_{\alpha,Q}$ 				&$\log T_\lambda \sim \lambda^{-\frac{p-1}{p}}$\\\hline
    $p>p_{\alpha,Q}$ 				&$T_\lambda =\infty$\\\hline
  \end{tabular}
 \end{center}
\end{table}

\newpage

\section*{Acknowledgments.}

The first-named author was supported by the research grant ARC DP220100285 from the Australian Research Council.

\section*{Statements and Declarations.}
\begin{itemize}
    \item The second-named author did not receive support from any organization for the submitted work.
    \item The authors have no relevant financial or non-financial interests to disclose.
\end{itemize}
  

\begin{bibdiv}
\begin{biblist}		
\bib{A01}{article}{
   author={D'Ambrosio, Lorenzo},
   title={Critical degenerate inequalities on the Heisenberg group},
   journal={Manuscripta Math.},
   volume={106},
   date={2001},
   number={4},
   pages={519--536},
   issn={0025-2611},
   review={\MR{1875346}},
   doi={10.1007/s229-001-8031-2},
}
\bib{AJS15}{article}{
   author={Azman, Ibtehal},
   author={Jleli, Mohamed},
   author={Samet, Bessem},
   title={Blow-up of solutions to parabolic inequalities in the Heisenberg
   group},
   journal={Electron. J. Differential Equations},
   date={2015},
   pages={No. 167, 9},
   review={\MR{3375998}},
}
\bib{BP85}{article}{
   author={Baras, Pierre},
   author={Pierre, Michel},
   title={Crit\`ere d'existence de solutions positives pour des \'equations
   semi-lin\'eaires non monotones},
   language={French, with English summary},
   journal={Ann. Inst. H. Poincar\'e{} Anal. Non Lin\'eaire},
   volume={2},
   date={1985},
   number={3},
   pages={185--212},
   issn={0294-1449},
   review={\MR{0797270}},
}
\bib{BHQ24}{article}{
   author={Bui, The Anh},
   author={Hong, Qing},
   author={Hu, Guorong},
   title={Generalized {Schr{\"o}dinger} operators on the {Heisenberg} group and {Hardy} spaces},
   journal={J. Funct. Anal.},
   volume={286},
   date={2024},
   number={10},
   pages={53},
   issn={0022-1236},
   review={\MR{1875346}},
   doi={10.1007/s229-001-8031-2},
}
\bib{Cohn}{book}{
   author={Cohn, Donald L.},
   title={Measure theory},
   series={Birkh\"auser Advanced Texts: Basler Lehrb\"ucher. [Birkh\"auser
   Advanced Texts: Basel Textbooks]},
   edition={2},
   publisher={Birkh\"auser/Springer, New York},
   date={2013},
   pages={xxi+457},
   isbn={978-1-4614-6955-1},
   isbn={978-1-4614-6956-8},
   review={\MR{3098996}},
   doi={10.1007/978-1-4614-6956-8},
}


\bib{CD}{article}{
   author={Coulhon, Thierry},
   author={Duong, Xuan Thinh},
   title={Maximal regularity and kernel bounds: observations on a theorem by
   Hieber and Pr\"uss},
   journal={Adv. Differential Equations},
   volume={5},
   date={2000},
   number={1-3},
   pages={343--368},
   issn={1079-9389},
   review={\MR{1734546}},
}
\bib{Da}{book}{
   author={Davies, E. B.},
   title={Heat kernels and spectral theory},
   series={Cambridge Tracts in Mathematics},
   volume={92},
   publisher={Cambridge University Press, Cambridge},
   date={1990},
   pages={x+197},
   isbn={0-521-40997-7},
   review={\MR{1103113}},
}
\bib{FRT24}{article}{
   author={Fino, Ahmad Z.},
   author={Ruzhansky, Michael},
   author={Torebek, Berikbol T.},
   title={Fujita-type results for the degenerate parabolic equations on the
   Heisenberg groups},
   journal={NoDEA Nonlinear Differential Equations Appl.},
   volume={31},
   date={2024},
   number={2},
   pages={Paper No. 19, 37},
   issn={1021-9722},
   review={\MR{4691932}},
   doi={10.1007/s00030-023-00907-2},
}
\bib{Folland}{article}{
   author={Folland, G. B.},
   title={Subelliptic estimates and function spaces on nilpotent Lie groups},
   journal={Ark. Mat.},
   volume={13},
   date={1975},
   number={2},
   pages={161--207},
   issn={0004-2080},
   review={\MR{0494315}},
   doi={10.1007/BF02386204},
}
\bib{FS}{book}{
   author={Folland, G. B.},
   author={Stein, Elias M.},
   title={Hardy spaces on homogeneous groups},
   series={Mathematical Notes},
   volume={28},
   publisher={Princeton University Press, Princeton, NJ; University of Tokyo
   Press, Tokyo},
   date={1982},
   pages={xii+285},
   isbn={0-691-08310-X},
   review={\MR{0657581}},
}
\bib{FHIL23}{article}{
   author={Fujishima, Yohei},
   author={Hisa, Kotaro},
   author={Ishige, Kazuhiro},
   author={Laister, Robert},
   title={Solvability of superlinear fractional parabolic equations},
   journal={J. Evol. Equ.},
   volume={23},
   date={2023},
   number={1},
   pages={Paper No. 4, 38},
   issn={1424-3199},
   review={\MR{4520263}},
   doi={10.1007/s00028-022-00853-z},
}
\bib{FHIL24}{article}{
   author={Fujishima, Yohei},
   author={Hisa, Kotaro},
   author={Ishige, Kazuhiro},
   author={Laister, Robert},
   title={Local solvability and dilation-critical singularities of
   supercritical fractional heat equations},
   language={English, with English and French summaries},
   journal={J. Math. Pures Appl. (9)},
   volume={186},
   date={2024},
   pages={150--175},
   issn={0021-7824},
   review={\MR{4745503}},
   doi={10.1016/j.matpur.2024.04.005},
}
\bib{F}{article}{
   author={Fujita, Hiroshi},
   title={On the blowing up of solutions of the Cauchy problem for
   $u\sb{t}=\Delta u+u\sp{1+\alpha }$},
   journal={J. Fac. Sci. Univ. Tokyo Sect. I},
   volume={13},
   date={1966},
   pages={109--124 (1966)},
   issn={0368-2269},
   review={\MR{0214914}},
}
\bib{GP21}{article}{
   author={Georgiev, Vladimir},
   author={Palmieri, Alessandro},
   title={Lifespan estimates for local in time solutions to the semilinear
   heat equation on the Heisenberg group},
   journal={Ann. Mat. Pura Appl. (4)},
   volume={200},
   date={2021},
   number={3},
   pages={999--1032},
   issn={0373-3114},
  review={\MR{4242116}},
   doi={10.1007/s10231-020-01023-z},
}
\bib{G}{article}{
   author={Grigor'yan, Alexander},
   title={Heat kernels and function theory on metric measure spaces},
   conference={
      title={Heat kernels and analysis on manifolds, graphs, and metric
      spaces},
      address={Paris},
      date={2002},
   },
   book={
      series={Contemp. Math.},
      volume={338},
      publisher={Amer. Math. Soc., Providence, RI},
   },
   isbn={0-8218-3383-9},
   date={2003},
   pages={143--172},
   review={\MR{2039954}},
   doi={10.1090/conm/338/06073},
}
\bib{H05}{article}{
   author={Han, Junqiang},
   title={Degenerate evolution inequalities on groups of Heisenberg type},
   journal={J. Partial Differential Equations},
   volume={18},
   date={2005},
   number={4},
   pages={341--354},
   issn={1000-940X},
   review={\MR{2188233}},
}
\bib{H73}{article}{
   author={Hayakawa, Kantaro},
   title={On nonexistence of global solutions of some semilinear parabolic
   differential equations},
   journal={Proc. Japan Acad.},
   volume={49},
   date={1973},
   pages={503--505},
   issn={0021-4280},
   review={\MR{0338569}},
}
\bib{HS}{article}{
   author={Hebisch, W.},
   author={Saloff-Coste, L.},
   title={On the relation between elliptic and parabolic Harnack
   inequalities},
   language={English, with English and French summaries},
   journal={Ann. Inst. Fourier (Grenoble)},
   volume={51},
   date={2001},
   number={5},
   pages={1437--1481},
   issn={0373-0956},
   review={\MR{1860672}},
   doi={10.5802/aif.1861},
}
\bib{H24}{article}{
   author={Hisa, Kotaro},
   title={Optimal singularities of initial data of a fractional semilinear heat equation in open sets},
   journal={arXiv: 2312.10969.},
}
\bib{HI18}{article}{
   author={Hisa, Kotaro},
   author={Ishige, Kazuhiro},
   title={Existence of solutions for a fractional semilinear parabolic
   equation with singular initial data},
   journal={Nonlinear Anal.},
   volume={175},
   date={2018},
   pages={108--132},
   issn={0362-546X},
   review={\MR{3830724}},
   doi={10.1016/j.na.2018.05.011},
}
\bib{HIT23}{article}{
   author={Hisa, Kotaro},
   author={Ishige, Kazuhiro},
   author={Takahashi, Jin},
   title={Initial traces and solvability for a semilinear heat equation on a
   half space of $\mathbb{R}^N$},
   journal={Trans. Amer. Math. Soc.},
   volume={376},
   date={2023},
   number={8},
   pages={5731--5773},
   issn={0002-9947},
   review={\MR{4630758}},
   doi={10.1090/tran/8922},
}
\bib{HS24}{article}{
   author={Hisa, Kotaro},
   author={Sier\.z\k{e}ga, Miko\l aj},
   title={Existence and nonexistence of solutions to the Hardy parabolic
   equation},
   journal={Funkcial. Ekvac.},
   volume={67},
   date={2024},
   number={2},
   pages={149--174},
   issn={0532-8721},
   review={\MR{4787124}},
}
\bib{HT21}{article}{
   author={Hisa, Kotaro},
   author={Takahashi, Jin},
   title={Optimal singularities of initial data for solvability of the Hardy
   parabolic equation},
   journal={J. Differential Equations},
   volume={296},
   date={2021},
   pages={822--848},
   issn={0022-0396},
   review={\MR{4278116}},
   doi={10.1016/j.jde.2021.06.011},
}
\bib{IS19}{article}{
   author={Ikeda, Masahiro},
   author={Sobajima, Motohiro},
   title={Remark on upper bound for lifespan of solutions to semilinear
   evolution equations in a two-dimensional exterior domain},
   journal={J. Math. Anal. Appl.},
   volume={470},
   date={2019},
   number={1},
   pages={318--326},
   issn={0022-247X},
   review={\MR{3865139}},
   doi={10.1016/j.jmaa.2018.10.004},
}	

\bib{IKO20}{article}{
   author={Ishige, Kazuhiro},
   author={Kawakami, Tatsuki},
   author={Okabe, Shinya},
   title={Existence of solutions for a higher-order semilinear parabolic
   equation with singular initial data},
   journal={Ann. Inst. H. Poincar\'e{} C Anal. Non Lin\'eaire},
   volume={37},
   date={2020},
   number={5},
   pages={1185--1209},
   issn={0294-1449},
   review={\MR{4138231}},
   doi={10.1016/j.anihpc.2020.04.002},
}		
\bib{IKS16}{article}{
   author={Ishige, Kazuhiro},
   author={Kawakami, Tatsuki},
   author={Sier\.z\c ega, Miko\l aj},
   title={Supersolutions for a class of nonlinear parabolic systems},
   journal={J. Differential Equations},
   volume={260},
   date={2016},
   number={7},
   pages={6084--6107},
   issn={0022-0396},
   review={\MR{3456827}},
   doi={10.1016/j.jde.2015.12.031},
}
\bib{JS}{article}{
   author={Jerison, David S.},
   author={S\'anchez-Calle, Antonio},
   title={Estimates for the heat kernel for a sum of squares of vector
   fields},
   journal={Indiana Univ. Math. J.},
   volume={35},
   date={1986},
   number={4},
   pages={835--854},
   issn={0022-2518},
   review={\MR{0865430}},
   doi={10.1512/iumj.1986.35.35043},
}
\bib{JKS16}{article}{
   author={Jleli, Mohamed},
   author={Kirane, Mokhtar},
   author={Samet, Bessem},
   title={A Fujita-type theorem for a multitime evolutionary $p$-Laplace
   inequality in the Heisenberg group},
   journal={Electron. J. Differential Equations},
   date={2016},
   pages={Paper No. 303, 8},
   review={\MR{3604748}},
}
\bib{KST77}{article}{
   author={Kobayashi, Kusuo},
   author={Sirao, Tunekiti},
   author={Tanaka, Hiroshi},
   title={On the growing up problem for semilinear heat equations},
   journal={J. Math. Soc. Japan},
   volume={29},
   date={1977},
   number={3},
   pages={407--424},
   issn={0025-5645},
   review={\MR{0450783}},
   doi={10.2969/jmsj/02930407},
}
\bib{KY94}{article}{
   author={Kozono, Hideo},
   author={Yamazaki, Masao},
   title={Semilinear heat equations and the Navier-Stokes equation with
   distributions in new function spaces as initial data},
   journal={Comm. Partial Differential Equations},
   volume={19},
   date={1994},
   number={5-6},
   pages={959--1014},
   issn={0360-5302},
   review={\MR{1274547}},
   doi={10.1080/03605309408821042},
}	
\bib{LS21}{article}{
   author={Laister, Robert},
   author={Sier\.z\polhk ega, Miko\l aj},
   title={A blow-up dichotomy for semilinear fractional heat equations},
   journal={Math. Ann.},
   volume={381},
   date={2021},
   number={1-2},
   pages={75--90},
   issn={0025-5831},
   review={\MR{4322607}},
   doi={10.1007/s00208-020-02078-2},
}	
\bib{LN92}{article}{
   author={Lee, Tzong-Yow},
   author={Ni, Wei-Ming},
   title={Global existence, large time behavior and life span of solutions
   of a semilinear parabolic Cauchy problem},
   journal={Trans. Amer. Math. Soc.},
   volume={333},
   date={1992},
   number={1},
   pages={365--378},
   issn={0002-9947},
   review={\MR{1057781}},
   doi={10.2307/2154114},
}
\bib{MPS}{article}{
   author={Maalaoui, Ali},
   author={Pinamonti, Andrea},
   author={Speight, Gareth},
   title={Function spaces via fractional Poisson kernel on Carnot groups and
   applications},
   journal={J. Anal. Math.},
   volume={149},
   date={2023},
   number={2},
   pages={485--527},
   issn={0021-7670},
   review={\MR{4594398}},
   doi={10.1007/s11854-022-0255-y},
}
\bib{MP01}{article}{
   author={Mitidieri, \`E.},
   author={Pokhozhaev, S. I.},
   title={A priori estimates and the absence of solutions of nonlinear
   partial differential equations and inequalities},
   language={Russian, with English and Russian summaries},
   journal={Tr. Mat. Inst. Steklova},
   volume={234},
   date={2001},
   pages={1--384},
   issn={0371-9685},
   translation={
      journal={Proc. Steklov Inst. Math.},
      date={2001},
      number={3(234)},
      pages={1--362},
      issn={0081-5438},
   },
   review={\MR{1879326}},
}
\bib{P98}{article}{
   author={Pascucci, Andrea},
   title={Semilinear equations on nilpotent Lie groups: global existence and
   blow-up of solutions},
   journal={Matematiche (Catania)},
   volume={53},
   date={1998},
   number={2},
   pages={345--357 (1999)},
   issn={0373-3505},
   review={\MR{1710767}},
}
\bib{P99}{article}{
   author={Pascucci, Andrea},
   title={Fujita type results for a class of degenerate parabolic operators},
   journal={Adv. Differential Equations},
   volume={4},
   date={1999},
   number={5},
   pages={755--776},
   issn={1079-9389},
   review={\MR{1696353}},
}
\bib{PV00}{article}{
   author={Pohozaev, Stanislav},
   author={V\'eron, Laurent},
   title={Nonexistence results of solutions of semilinear differential
   inequalities on the Heisenberg group},
   journal={Manuscripta Math.},
   volume={102},
   date={2000},
   number={1},
   pages={85--99},
   issn={0025-2611},
   review={\MR{1771229}},
   doi={10.1007/PL00005851},
}
\bib{QS19}{book}{
   author={Quittner, Pavol},
   author={Souplet, Philippe},
   title={Superlinear parabolic problems},
   series={Birkh\"auser Advanced Texts: Basler Lehrb\"ucher. [Birkh\"auser
   Advanced Texts: Basel Textbooks]},
   edition={2},
   note={Blow-up, global existence and steady states},
   publisher={Birkh\"auser/Springer, Cham},
   date={2019},
   pages={xvi+725},
   isbn={978-3-030-18220-5},
   isbn={978-3-030-18222-9},
   review={\MR{3967048}},
   doi={10.1007/978-3-030-18222-9},
}
\bib{RS13}{article}{
   author={Robinson, James C.},
   author={Sier\.z\polhk ega, Miko\l aj},
   title={Supersolutions for a class of semilinear heat equations},
   journal={Rev. Mat. Complut.},
   volume={26},
   date={2013},
   number={2},
   pages={341--360},
   issn={1139-1138},
   review={\MR{3068603}},
   doi={10.1007/s13163-012-0108-9},
}
\bib{RY22}{article}{
   author={Ruzhansky, Michael},
   author={Yessirkegenov, Nurgissa},
   title={Existence and non-existence of global solutions for semilinear
   heat equations and inequalities on sub-Riemannian manifolds, and Fujita
   exponent on unimodular Lie groups},
   journal={J. Differential Equations},
   volume={308},
   date={2022},
   pages={455--473},
   issn={0022-0396},
   review={\MR{4340784}},
   doi={10.1016/j.jde.2021.10.058},
}
\bib{Simon}{book}{
   author={Simon, Leon},
   title={Lectures on geometric measure theory},
   series={Proceedings of the Centre for Mathematical Analysis, Australian
   National University},
   volume={3},
   publisher={Australian National University, Centre for Mathematical
   Analysis, Canberra},
   date={1983},
   pages={vii+272},
   isbn={0-86784-429-9},
   review={\MR{0756417}},
}
\bib{St}{book}{
   author={Strauss, Walter A.},
   title={Partial differential equations},
   note={An introduction},
   publisher={John Wiley \& Sons, Inc., New York},
   date={1992},
   pages={xii+425},
   isbn={0-471-54868-5},
   review={\MR{1159712}},
}
\bib{S75}{article}{
   author={Sugitani, Sadao},
   title={On nonexistence of global solutions for some nonlinear integral
   equations},
   journal={Osaka Math. J.},
   volume={12},
   date={1975},
   pages={45--51},
   issn={0388-0699},
   review={\MR{0470493}},
}
\bib{Y21}{article}{
   author={Yang, Zhipeng},
   title={Fujita exponent and nonexistence result for the Rockland heat
   equation},
   journal={Appl. Math. Lett.},
   volume={121},
   date={2021},
   pages={Paper No. 107386, 6},
   issn={0893-9659},
   review={\MR{4260530}},
   doi={10.1016/j.aml.2021.107386},
}
\bib{Z98}{article}{
   author={Zhang, Qi S.},
   title={The critical exponent of a reaction diffusion equation on some {Lie} groups},
   journal={Math. Z.},
   volume={228},
   date={1998},
   number={1},
   pages={51--72},
   issn={0025-5874},
   review={\MR{4630758}},
   doi={10.1007/PL00004602},
}
\bib{Zhang}{article}{
   author={Zhang, Qi S.},
   title={A sharp comparison result concerning Schr\"odinger heat kernels},
   journal={Bull. London Math. Soc.},
   volume={35},
   date={2003},
   number={4},
   pages={461--472},
   issn={0024-6093},
   review={\MR{1978999}},
   doi={10.1112/S002460930300211X},
}


\end{biblist}
\end{bibdiv}
\end{document}